\documentclass[a4paper, 10pt, reqno, final]{article}

\usepackage[T1]{fontenc}		     
\usepackage[utf8]{inputenc}			 
\usepackage[american]{babel}
\usepackage{float}
\usepackage{caption}
\usepackage{subcaption}

\usepackage{amsmath}
\usepackage{amssymb}
\usepackage{amsthm}
	\theoremstyle{plain}
		\newtheorem{mainthm}{\textsc{Theorem}}		
				
				\newtheorem{thm}{Theorem}[section]	
						\newtheorem{cor}[thm]{Corollary}	
				
				\newtheorem{lem}[thm]{Lemma}		
						\newtheorem{prop}[thm]{Proposition}

	\theoremstyle{definition}
		\newtheorem{defn}[thm]{Definition}	
						
					\theoremstyle{remark}
		\newtheorem{rem}[thm]{Remark}		
				\newtheorem{note}[thm]{Notation}		
				\numberwithin{equation}{section}
				\usepackage{mathtools}		
				\usepackage{mathrsfs}		
				
\usepackage{eucal}			

\usepackage{braket}			

\usepackage[all,pdf]{xy}		

\usepackage{mparhack}		
\usepackage{relsize}			

\usepackage{a4wide}		

\usepackage{booktabs}		
\usepackage{multirow}		
\usepackage{caption}		
\usepackage{rotating}

\usepackage{varioref}		
\usepackage{footmisc}

\usepackage{graphicx}		
\usepackage{epsfig}

\usepackage{enumerate}		
\usepackage[colorlinks=true, linkcolor=black, citecolor=black,
urlcolor=blue]{hyperref}		
\mathtoolsset{showonlyrefs}

\newcommand{\trasp}[1]{{#1}^\mathsf{T}}	
\newcommand{\R}{\mathbb{R}}		
\newcommand{\ZZ}{\mathbf{Z}}
\newcommand{\C}{\mathbf{C}}		
\newcommand{\U}{\mathbf{U}}		
\newcommand{\MP}{\mathcal{P}}
\newcommand{\SO}{\mathrm{SO}}
\newcommand{\Sp}{\mathrm{Sp}}
\newcommand{\Gr}{\mathrm{Gr}}
\newcommand{\Mprod}[2]{ \left\langle {#1},{#2} \right\rangle_M}		
\newcommand{\N}{\mathbb{N}}		
\newcommand{\iCLM}{\mu^{\scriptscriptstyle{\mathrm{CLM}}}}
\DeclareMathOperator{\sgn}{sgn}		
\renewcommand{\=}{\coloneqq}			
\newcommand{\email}[1]{\href{mailto:#1}{\textsf{#1}}}

\newcommand{\Id}{I}

\title{Morse index of circular solutions for attractive central force problems on
surfaces}
\author{Stefano Baranzini, Alessandro Portaluri
\thanks{The
author is partially supported by Progetto di Ricerca GNAMPA - INdAM, codice CUP\_E55F22000270001 “Dinamica simbolica e soluzioni periodiche per problemi singolari della Meccanica Celeste”}, Ran Yang\thanks{The author is supported by the National Natural Science Foundation of China (No. 12001098), the Doctoral research start-up fund of East China University of Technology (No. DHBK2019204) and the State Scholarship Fund from the CSC.} }
\date{\today}

\date{\today}
\begin{document}
 \maketitle

\begin{abstract}
The classical theory of attractive central force problem on the standard (flat) Euclidean  plane $\R^2$ can be generalized to surfaces by reformulating the basic underlying physical principles by means of differential geometry. Attractive central force problems on state manifolds appear quite often and in several different context ranging from nonlinear control theory to mobile robotics, thermodynamics, artificial intelligence, signal transmission and processing and so on.  

The aim of the present paper is to analyze the variational properties of the circular periodic orbits in the case of attractive power-law potentials of the Riemannian distance on revolution's surfaces.  We compute the stability properties  and the Morse index by developing a suitable intersection index in the Lagrangian Grassmannian and symplectic context. 
\end{abstract}


\tableofcontents

\section{Introduction, description of the problem and main results}\label{sec:intro}

Nonlinear dynamical systems on  manifolds appear quite often in
mathematical physics  They represent the natural mathematical framework for modeling  real phenomena in mobile robotics, mathematical optimization, thermodynamics and so on. 

 Newtonian mechanics on non-flat spaces can be  formulated in the language of Riemannian geometry.  Mechanical system  can be described as triple $(M,g,V)$ where $M$ is a smooth manifold representing the configuration space, $g$ the metric tensor (determining the kinetic energy) and $V$ the potential.  

 In Physics and Classical Mechanics, many interactions are modeled using potentials depending on the distance alone. Both when one studies systems of many particles interacting with each other or when one considers a single particle that interacts with a source. Thus, it seems natural to investigate systems on manifolds $M$ whose potential is a function of the distance from a point.


\subsection*{Historical comments}

In the standard  Euclidean space a particularly interesting attractive central force potential is provided by the gravitational Keplerian potential. In the last decades, several authors  provided a generalization of the gravitational Keplerian potential in the constant curvature case, starting with the well-known manuscript of Harin \& Kozlov \cite{HK92}.  One of the primary motivation of the aforementioned paper was   to find  classes of central force potentials on constant curvature   spaces for which the two fundamental properties of the gravitational central field still hold:
\begin{itemize}
\item[-] The potential is an harmonic function (for the 3-dimensional models)

\item[-] All bounded orbits are closed (Bertrand's theorem).  
\end{itemize}
Replacing the Laplace's equation by the Laplace-Beltrami equation, authors showed that on non-vanishing constant curvature case the gravitational potential energy, up to a constant,  has to be replaced in spherical coordinates by
\begin{itemize}
\item (Sphere case) 	$V(\theta)= \cot \theta$ where $\theta\in(-\pi/2, \pi/2)$ 
\item (Hyperbolic plane case) $V(\theta)= \coth \theta$ where $\theta\in(-\pi/2, \pi/2)$.
\end{itemize}
Moreover, in analogy with the Euclidean case, any solution is confined to a totally geodesic surface (either a sphere or a plane). 
In this respect we want to mention another interesting integrable generalization of the Kepler problem, nowadays known as the MICZ-Kepler system. The origin of this generalization can be traced back to McIntosh, Cisneros and Zwanziger in 1968. (Cfr. \cite{Mon13} and references therein for further details). 

By means of these cotangent and hyperbolic cotangent potentials, the classical gravitational $n$-body problem has been formulated in these new geometrical context. (We refer the interested reader to \cite{DPS12a, DPS12b} and references therein).  

A remarkable  variational result in the context of gravitational central force potential has been proved by W. Gordon in \cite{Gor77}   where is shown that periodic elliptic solutions of the Kepler problem are minimizers of the Lagrangian action functional. In particular, their Morse index is zero.   Recently, in \cite{DDZ19},  using the generalized Conley-Zehnder index and a classical  index theorem, authors generalized this result   to constant curvature surfaces and for the classes of cotangent potentials discussed above. 

Despite  these results and generalizations  to constant curvature surfaces, inspired by \cite{BK17} and motivated by the application in classical mechanics (cfr. for instance \cite{Fio17} and references therein), in the present paper  we are interested  in investigating the variational properties of circular orbits of an attractive central force problem in which the potential energy is a function of the Riemannian distance  from a point (notice that the cotangent potential has two singularities). Among all functions of the Riemannian distance, a natural class is provided by power-law function.

\subsection*{Description of the problem and main results}

 In the rest of the paper, we will work on the plane $\R^2$ with a Riemannian metric which is conformal to the flat one. Furthermore, we will assume that this conformal factor depends solely on the distance from a fixed point. This is equivalent to considering the induced Riemannian metric of revolution surfaces embedded  $\R^3$, as explained in  Proposition~\ref{thm:fava1}.   This symmetry will be crucial for us, it implies conservation of the angular momentum and existence of families of circular solutions. In the second part of the paper we will focus more specifically on the two constant curvature surface: the sphere or hyperbolic plane. 

 In polar coordinates $\xi,\theta$ and up to rescaling time and normalizing the radial variable, we  end up considering the  following  Lagrangian function 
\begin{equation}
L_\alpha(\xi, \vartheta,\dot \xi,\dot \vartheta)=\dfrac12 p(\xi)\big[\dot \xi^2+\xi^2\dot \vartheta^2\big]+q(\xi).
\end{equation} 
Here  $p$ denotes the conformal factor  and $q$ the potential. For the sphere and the hyperbolic plane with constant curvature metrics the conformal factors potentials we will consider are
\begin{itemize}
	\item (Sphere case)\begin{equation}
			\label{eq:def_coeff_sphere}
			p(\xi):= \dfrac{2}{(1+\xi^2)^{2}}, \qquad q(\xi):= \arctan^\alpha \xi. 
	\end{equation}
	\item (Hyperbolic plane case)
	\begin{equation}
		\label{eq:def_coeff_hyp}
		p(\xi):= \dfrac{2}{(1-\xi^2)^{2}}, \qquad q(\xi):=\ln^\alpha\left(\dfrac{1+\xi}{1-\xi}\right). 
\end{equation}
\end{itemize}
(Here $\alpha$ is non-vanishing and $\xi$ is non-negative). $T$-periodic solutions  of the associated Euler-Lagrange equation can be seen as critical points of the Lagrangian action function 
\[
A(x)=\int_0^T L\big(\xi(t), \vartheta(t), \xi'(t), \vartheta'(t)\big) \, dt
\]
where $x=(\xi, \vartheta)$,  $T>0$ denotes the prime period of the orbit and $A$ is defined on the Hilbert  space of the $H^1$  loops (of period $T$) in the punctured plane. We let $x$ be a $T$-periodic circular  orbit, i.e. a solution of the form $x(t) = (\xi_0, \theta_0 + t \omega)$ for $t \in [0,T/\omega]$. By linearizing the Euler-Lagrangian equation along such a circular solution, we can associate the following Hamiltonian system 
\begin{equation}\label{eq:linear-system-general-intro-0}
\dot z(t)=A (\xi_0) z(t)\qquad  t\in [0,T].
\end{equation} 
where $A(\xi_0)$ is the Hamiltonian matrix given by 
\[
A (\xi_0)=\begin{bmatrix}
0&b (\xi_0)&d(\xi_0)&0\\0&0&0&0\\a(\xi_0)&0&0&0\\0&c(\xi_0)&-b(\xi_0)&0
\end{bmatrix}
\]
$a, b, c, d $   are real functions depending on $\xi$  and $a$ is positive. Denoting by  $m^-(x)$ the Morse index of the critical point $x$, our first results reads as follows. 

\begin{mainthm}\label{thm:main1-intro}
 Let $x=\big(\xi_0, \vartheta(t)\big)$ be a circular $T$-periodic solution of the Euler-Lagrange equation.
 Then, the Morse index of $x$ is given by 
 \begin{equation}
m^-(x)=\begin{cases}
2k &\textrm{ if }\quad d<0 \textrm{ and }cd+b^2\ge  0\\[10pt]
2k+1 &\textrm{ if }\quad d<0 \textrm{ and } cd+b^2< 0\\[10pt]
0 &\textrm{ if }\quad \begin{cases}
d=0, \ b=0 \textrm{ and }c>0\\
d=0 \textrm{ and } b\neq 0\\ 
d>0 \textrm{ and } cd+b^2>0
\end{cases}
\end{cases}
\end{equation}
where $k\in\N$ is given by $k\cdot 2\pi<\sqrt{-ad (\xi_0)} \cdot T\le (k+1)\cdot 2\pi$.
\end{mainthm}
The proof of this result  will be given in Section~\ref{sec:general-index}. As an application, in Section~\ref{sec:computations of maslov index}, we provide a precise  study of the index in the case of constant curvature surfaces. Here you can find a pictorial version of these results and we refer to Section~\ref{sec:computations of maslov index} for the precise statements and proof. 
\begin{mainthm}{\bf (Sphere Case)}
Let $p$ and $q$ be as in \eqref{eq:def_coeff_sphere} and let $x$ be a circular solution. The Morse index varies between $0$ and $3$ as displayed in the following figure 
\begin{figure}[ht]
     \centering
     \begin{subfigure}[b]{0.4\textwidth}
         \centering
         \includegraphics[scale=0.15]{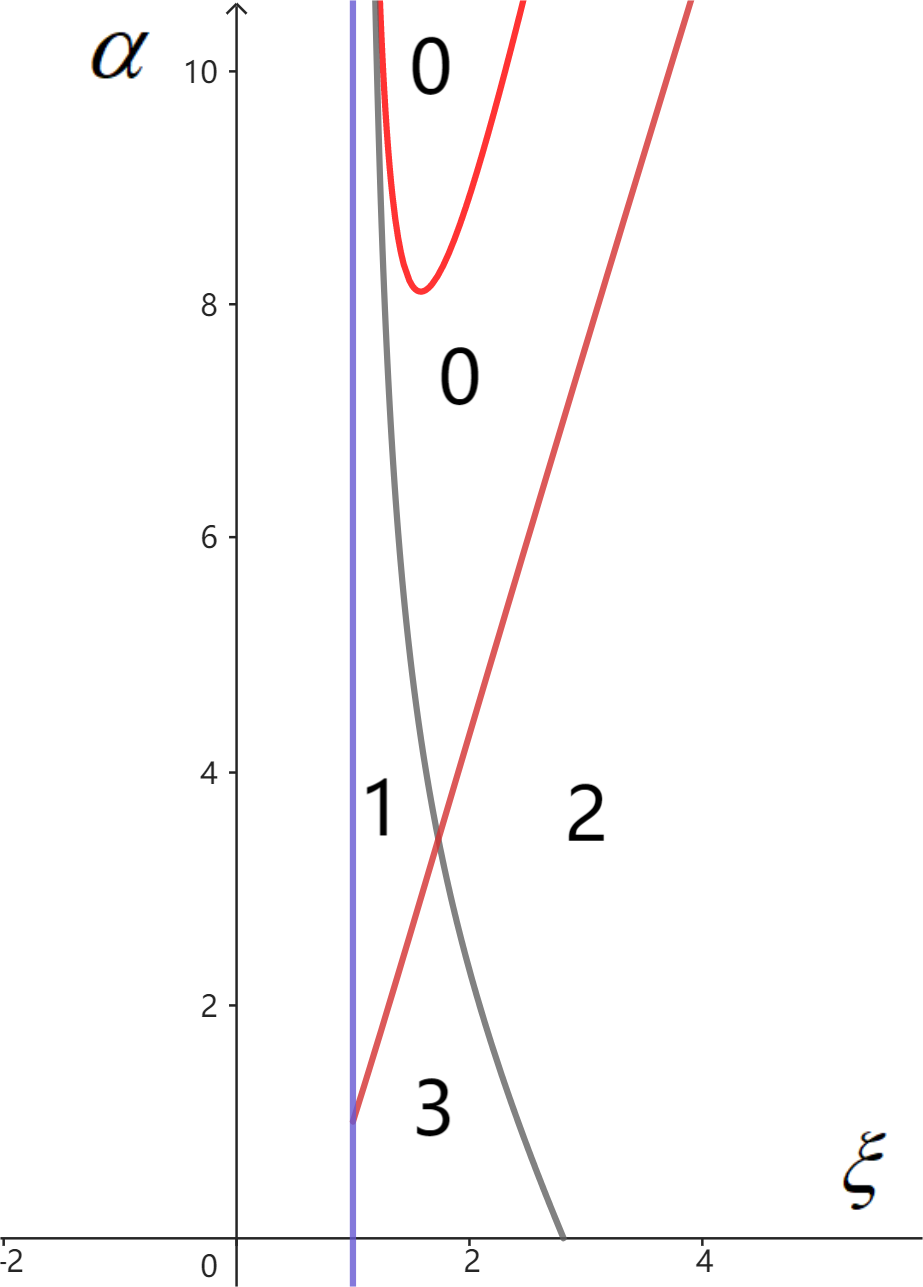}
         \caption{{\bf (Sphere case $\alpha$ positive)} In this figure are displayed the subregions of the $\widehat{\xi O\alpha}$-region $\Omega_1:=(1,+\infty)\times (0, +\infty)$ labeled by the Morse index of the corresponding circular orbit} 
	\label{fig:sphere-positive-alpha-intro-positive}
     \end{subfigure}
     \qquad
     \begin{subfigure}[b]{0.4\textwidth}
         \centering
        \includegraphics[scale=0.15]{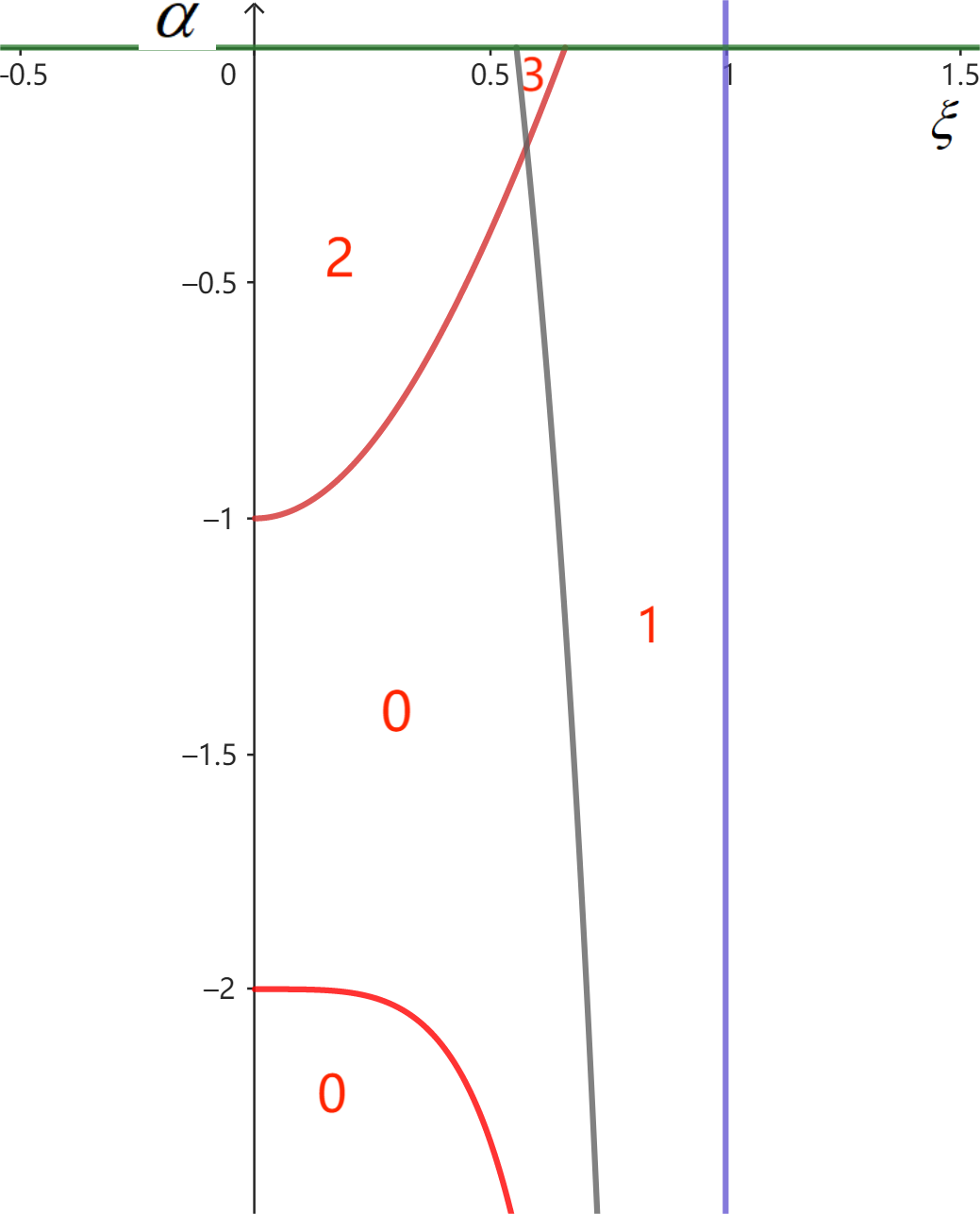}
         \caption{{\bf (Sphere case $\alpha$ negative)} In this figure are displayed the subregions of the $\widehat{\xi O\alpha}$-region $\Omega_2:=(1,+\infty)\times (-\infty, 0)$ labeled by the Morse index of the corresponding circular orbit} 
	\label{fig:sphere-negative-alpha-intro}
     \end{subfigure}
  \end{figure}
  \end{mainthm}
  
  \begin{mainthm} {\bf (Hyperbolic plane Case)}
 Let $p$ and $q$ be as in \eqref{eq:def_coeff_hyp}. For any $n \in \mathbb{N}$ there exists a circular solution $x_n$ with Morse index equal to $2n$.
In the following figure are displayed the first index jump regions.
\begin{figure}[ht]
	\centering
	\includegraphics[scale=0.10]{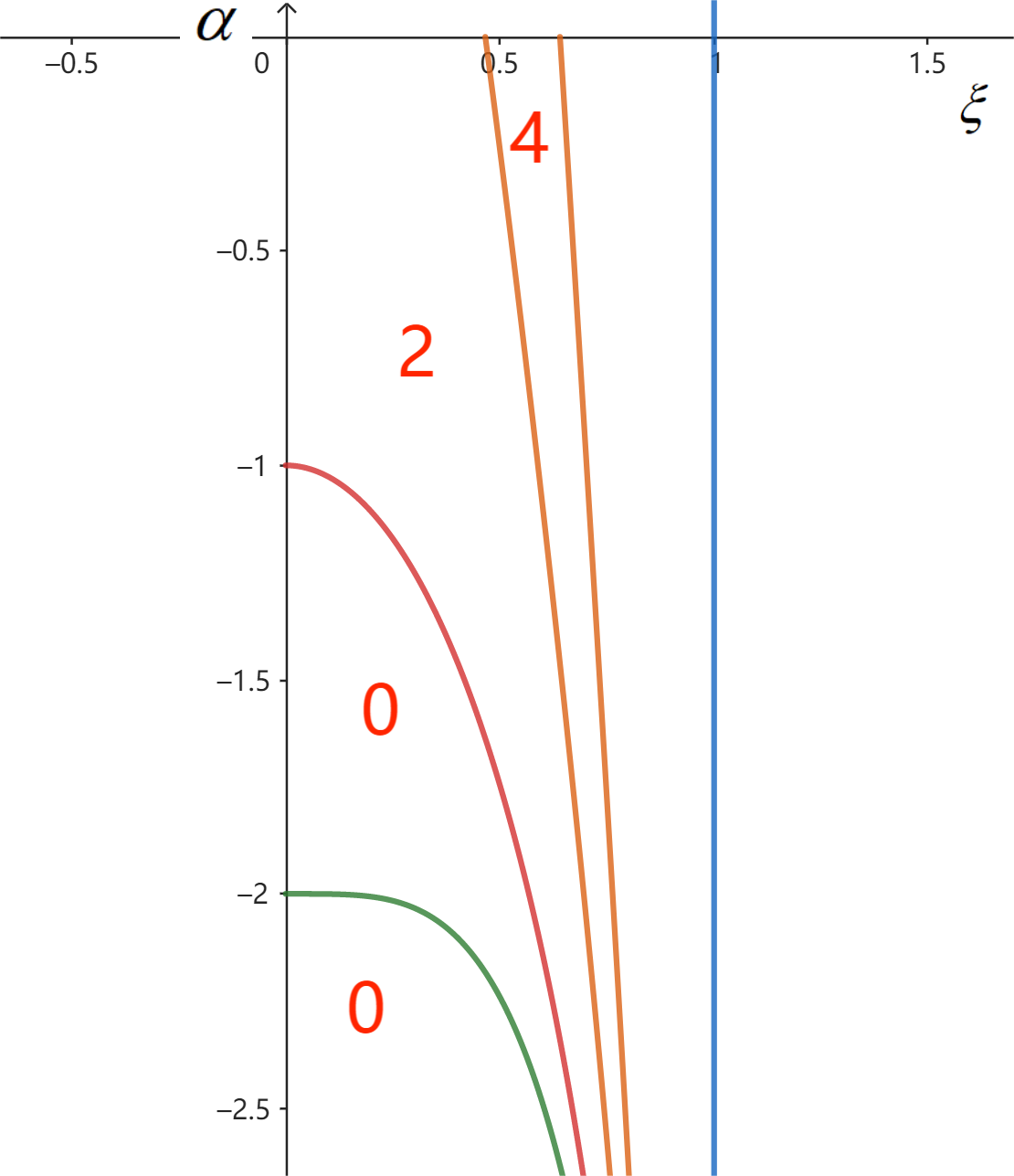}
		\caption{{\bf (Hyperbolic plane case)} In the figure we represent the subregions of the $\Omega_3$ corresponding to the jumps of the Morse index.}
	\label{fig:hyper-negative-alpha-intro}
\end{figure}  
\end{mainthm} 

We remark that in the Euclidean case we recover the results about the computation of the Morse index as well as an instability result for these class of solutions, already proved in \cite{BJP16} and \cite{KOP21}. 
\begin{mainthm}\label{pro:index-of-euclidean-intro}
	Let $p(\xi)=1$, $q(\xi)= \xi^\alpha$  and let $x$ be a circular solution. Then, the Morse index of $x$  is given by
	\begin{equation}
	m^-(x)=\begin{cases}
	0 &\textrm{ if }\quad \alpha\in(-\infty,-1]\\
	2  &\textrm{ if}\quad \alpha\in (-1,0)
	\end{cases}.
	\end{equation}	  
\end{mainthm}
\begin{rem}
The case $\alpha=-1$ corresponds to the classical singular gravitational potential. In this case, we have  
	\begin{equation}
	\iota_1(\gamma)=0
	\end{equation} 
and the corresponding circular orbit is  linearly unstable.	
\end{rem}


\section{Surfaces of revolutions, radial functions and geodesics}

  Central potentials are extremely important in physics. It is well known that a particle moving in the Euclidean space  and subject to a central force is confined to move in the plane spanned by its initial position and velocity. From the mathematical point of view, this means that the system is completely integrable and the angular momentum is a conserved quantity. It is thus natural to look for systems (and geometries) in which these symmetries are still present. A big class of $2D$ examples is given by revolution surfaces. They posses a natural $\mathbb{S}^1$ symmetry and a corresponding $\mathbb{S}^1$ action by isometry. We will consider potentials invariant with respect to it, i.e. depending solely on the distance from the fixed points of this action.


\subsection{Radially conformal flat metrics on the plane}

In the Euclidean three dimensional space $(\mathbb E^3, g_e)$ equipped with a Cartesian reference frame, we denote a  point $P$  by the vector $(x_0, x_1, x_2)\in \R^3$. We let $\ell_i $ to be the straight line corresponding to the $i$-th coordinate and we denote by $\Pi_{i,j}$ the plane spanned by $\ell_i$ and $\ell_j$ for $i,j =0,1,2$.  

Given a smooth curve $\gamma:[-1,1] \to \mathbb{R}^3$, we assume 
\begin{enumerate}
		\item[(C1)] $\gamma$ is simple, namely $\gamma(t) = \gamma (s)$ iff $t=s$ and regular meaning that $\dot \gamma(t) \ne 0$, for every $t \in [-1,1]$
		\item[(C2)] $\gamma([-1,1])$ is contained in the $\Pi_{0,2}$-plane whose ends are two distinct points on the $\ell_2$-axis 
		\item[(C3)] $\dot \gamma_+(-1),\  \dot\gamma_-(1)$ belongs to $\Pi_{0,1}$-plane  where $\dot \gamma_+$ (resp. $\dot \gamma_-$) denotes the right (resp. left) derivative
		\end{enumerate}		
Because of (C2), in local coordinates $\gamma$ is pointwise defined by  $\gamma(t)=\big(x_0(t), 0 , x_2(t)\big)$ for $t \in [-1,1]$.
We set 
\begin{multline}
\Gamma:= \Set{\gamma \in \mathscr C^\infty([-1,1], \mathbb E^3)|(C1)-(C4) \textrm{ hold}} \quad \textrm{ and } 
\Sigma:=\Set{M_\gamma:=G\cdot \gamma| \gamma \in \Gamma}\\ \textrm{ where } $G$ \textrm{ denotes the group of isometries of $\mathbb E^3$ fixing the line $\ell_2$}.
\end{multline}
By the properties of $\gamma$, it  follows that $M_\gamma$ is diffeomorphic to the $2$-dimensional sphere  $\mathbb S^2$ and because of (C3), we get that $M_\gamma$ is a smooth surface. We refer to the curve $\gamma$ as the {\em generating curve} of $M_\gamma$ and the circles described by the points of $\gamma$ correspond to the {\em parallels} of $M_\gamma$.  We denote by $\widehat M_\gamma:=M_\gamma\setminus(\ell_2\cap M_\gamma)$. Denoting by $A$ the degenerate skew-symmetric matrix corresponding to the infinitesimal generator of the group $G$, given by $A=\begin{bmatrix}
0&-1&0\\1&0&0\\0&0&0	
\end{bmatrix}
$,
we get a surjective diffeomorphism between the open cylinder $\mathbb{S}^1\times (-1,1) $ and  $\widehat M_\gamma$ defined by 
\begin{equation*}
\phi :\mathbb{S}^1\times (-1,1)  \ni (t,\exp{(A\,\theta)}) \longmapsto \exp{(A\,\theta)}\, \gamma(t)\in \widehat M_\gamma.
\end{equation*}
By means of this diffeomorphism we define a Riemannian metric $\widehat g_\gamma$ on $\widehat M_\gamma$ by  pulling back the Euclidean Riemannian metric on  $\mathbb{E}^3$ to $\widehat M_\gamma$.
By an explicit calculations, we get that 
\begin{multline}
g_e(\exp{(A\,\theta)} \dot \gamma,\exp{(A\,\theta)} \dot \gamma) = |\dot \gamma|^2, \qquad   g_e(A\exp{(A\,\theta)} \gamma, A\exp{(A\,\theta)} \gamma) = g_e(A\gamma,A\gamma)=x_0(t)^2\\ g_e(\exp{(A\,\theta)} \dot \gamma,A\exp{(A\,\theta)} \gamma) = g_e(\dot \gamma, A\gamma)=0.
\end{multline}
As consequence of this discussion, we get that  the Riemannian metric $\widehat g_\gamma$ on $\widehat M_\gamma$  in these new cylindrical coordinates, is given by
\begin{equation*}
	\widehat g_\gamma = \phi^*g_e =x_0(t)^2d \theta^2+ \vert \dot\gamma\vert^2 dt^2
\end{equation*}
where, as usually  $\phi^*$ denotes the pull-back. 
\begin{prop}\label{thm:fava1}
There exists a radial function $r\mapsto a(r)$ on $\mathbb{R}^2 \setminus \{(0,0)\}$ such that  $( \widehat M_\gamma,  \widehat g_\gamma)$ is isometric to $(\R^2, a^2(r) g_e)$. 
\end{prop}
\begin{proof}
We start by observing that, repeating verbatim the previous discussion once  replaced $M_\gamma$ by  $\mathbb S^2$, we can construct the  diffeomorphism from the open cylinder $\mathbb{S}^1\times (-1,1) $ to $\widehat {\mathbb S^2}:=\mathbb S^2\setminus\{N,S\}$ where $N,S$ denote the north and the south poles respectively. Namely
\begin{equation*}
\phi :\mathbb{S}^1\times (-1,1)  \ni (t,\exp{(A\,\theta)}) \longmapsto \exp{(A\,\theta)}\, \widetilde \gamma(t)\in\mathbb S^2,
\end{equation*}
where $\widetilde \gamma$ is the curve defined by 
\begin{equation}
	\label{eq:profile_curve_sphere}
	\widetilde \gamma : [-1,1] \to \mathbb{R}^3, \quad \tilde \gamma(t) = \left(\sqrt{1-t^2},0,t\right).
\end{equation}
Denote by $\widehat {\mathbb R^2}:=\R^2\setminus\{(0,0)\}$ and let $\psi$ be  the inverse  of the stereographic projection from the north pole of $\mathbb S^2$. It is given explicitly by: 
\begin{equation*}
	\psi: \widehat {\mathbb R^2} \ni (u,v) \longmapsto \left(\dfrac{2u}{1+u^2+v^2},\dfrac{2v}{1+u^2+v^2},\dfrac{u^2+v^2-1}{1+u^2+v^2}\right) \in\widehat{\mathbb{S}^2}.
\end{equation*}
The map  $\widetilde \phi^{-1} \circ \psi: \R^2 \to (-1,1) \times \mathbb S^1$ with respect to the local system of coordinates $(u,v)$ is given by 
\[
\widetilde \phi^{-1}\circ \psi (u,v) = (t(u,v),\theta(u,v)) =  \left(\dfrac{u^2+v^2-1}{1+u^2+v^2}, \arctan\left(\dfrac{u}{v}\right)\right).  
\]
If $\sigma: \widehat {\mathbb R^2}\to\widehat {\mathbb R^2}$ is the polar change of coordinates defined by 
\[
\sigma (\rho, \alpha)=\big(\sqrt{u^2+v^2}, \arctan(u/v)\big)
\]
and denoting by $F$ the diffeomorphism given by 
\[
F= \widetilde \phi^{-1}\circ \psi\circ \sigma: \widehat {\mathbb R^2} \ni (\rho, \alpha) \longmapsto  (t(\rho,\alpha),\theta(\rho,\alpha))=\left(\dfrac{\rho^2-1}{1+\rho^2},\alpha\right) \in \mathbb{S}^1\times (-1,1)
\]
since $g_\gamma= \phi^* g_e$, we finally get 
\[
F^* \widehat g_\gamma = \left( \dfrac{4\rho^2}{(1+\rho^2)^2}\right)^2 \vert \dot \gamma(t(\rho)) \vert^2 d\rho^2+ x_0(t(\rho))^2 d \alpha^2.
\]
 To conclude the proof of the first claim, we introduce a suitable time rescaling along the curve $\gamma$.

Let $\lambda$ be a reparametrization of the time variable $t$ and let us denote by $'$ the new time variable $s$. Thus we get that  
\begin{multline}
	\widehat g_\gamma = |\gamma'(\lambda(s))|^2\lambda'(s)^2 ds^2+ x_0(\lambda(s))^2 d\theta^2\quad \textrm{ and} \\
	F^*\widehat g_\gamma =\dfrac{x_0(\lambda(s))^2}{\rho^2}\left[\left[\dfrac{4\rho^2}{(1+\rho^2)^2}\right]^2\dfrac{\lambda'^2|\gamma'(\lambda(s))|^2}{x_0(\lambda(s))^2} d\rho^2+ \rho^2d\alpha^2\right].
\end{multline}
Thus,  imposing that $F^*g_\gamma$ is conformal to the flat metric, uniquely determines  $\lambda$. For, we let 
\begin{equation*}
\lambda' (s(\rho)) = \left(\dfrac{1+\rho^2}{2 \rho}\right)^2\dfrac{x_0(\lambda(s(\rho)))}{\vert\gamma'(\lambda(s(\rho)))\vert} \quad \textrm{ where } \quad s(\rho) = \dfrac{\rho^2-1}{\rho^2+1}.
\end{equation*}
With such a choice it readily follows that 
\[
F^*\widehat g_\gamma =\dfrac{x(\lambda(s))^2}{\rho^2} \big(d\rho^2 + \rho^2 d\alpha^2\big).
\]
which concludes the proof.
\end{proof}


\subsection{Meridians and geodesics}

 On a surface of revolution, there are two families of distinguished geodesics: parallels and meridians. The first ones are orbits of the $\mathbb{S}^1$ action whereas the second ones are all the geodesics issuing from fixed points. The next Proposition shows this and provides and integral expression for the Riemannian distance.
\begin{prop}\label{thm:fava2}
	Let su consider $(\mathbb{R}^2,g)$ where 
	\begin{equation*}
			g := a^2(r)g_e, \text{ and } \quad a(r)>0
	\end{equation*}
    and let $X$ be the radial vector field given by $X:= a(r)^{-1}\partial_r$. Then we have 
	\[
	\nabla_XX =0
	\]
	and its integral curves are unit speed geodesics through the origin.  Moreover, solutions starting from the origin are minimizing geodesics. The (Riemannian) distance  of a point $x$ from origin is given by:
	\begin{equation}
		\label{eq:riemannian_distance}
		d_a(0,x) = \int_0^{|x|} {a(t) }dt 
	\end{equation}
	where $| \cdot |$ denotes the Euclidean norm.
		\end{prop}
	\begin{proof}
		Since $\vert X\vert_g =1$, its integral curves are unit speed curves. To check the geodesic condition we have to compute $\nabla_X X$. Standard differential formulas involving  conformal changes of the type $g = e^{2\log a} g_e$, of the metric yield:
		\begin{equation*}
			\nabla_X X = \nabla^e_XX+ 2 g_e(\nabla \log(a),X)X-g_e(X,X)\nabla \log(a)
		\end{equation*}
	where $\nabla^e$ denotes the Levi-Civita connection with respect to the Euclidean metric. A direct computation yields:
	\begin{equation*}
		\nabla^e_XX = a^{-2}(r)\nabla_{\partial_r} \partial_r +a^{-1}(\partial_r a^{-1}(r) )\partial_r = - (\partial_r a )a^{-3} \partial_r.
	\end{equation*}
Since $\nabla \log(a) = a^{-1} \partial_r a$ and $g_e(X,X) = a^{-2}$. Summing up the above computations, we get  $\nabla_X X =0$. By this we immediately get that $\gamma$ is a  unit speed geodesic.

Now, we prove that the straightlines from the origin are minimizing. Let $\eta$ be any curve joining $0$ to a point $x$ then, in coordinates $\theta$ and $r$ we can write $\eta = r(t) e^{i\theta(t)}$. It follows that 
\[
\vert\dot\eta\vert_g =a(r) \sqrt{\dot{r}+r^2\dot \theta}\ge a(r) \vert \dot 
r \vert. 
\]
On the other hand, if $x= r_0 e^{i\theta_0}$, the curve $\tilde \eta(t) = r(t)e^{i\theta_0}$ satisfies the same boundary conditions but $\ell(\eta) \ge \ell(\tilde\eta)$. It follows that minimizing geodesic from $0$ to the point $x$ are straight lines. 

Take an integral line $\eta$ of $X$ joining $0$ to a point $x$. Since $X$ gives arclength parametrized geodesics and since along a straight geodesic starting at the origin, we get that $|\eta|_g= a(r) \dot r=1$ we have by integrating over $[0,T]$:
\begin{equation*}
a(r) \dot{r} = 1 \Rightarrow d(0,x) = T = \int_0^Ta(r)\dot{r} dt \underset{[s =r(t)]}{=}\int_0^{r(T)}a(s)ds = \int_0^{|x|}a(s)ds.
\end{equation*}
\end{proof}

One direct consequence of Proposition~\ref{thm:fava2}, is that we can, in principle,  rewrite any  radial function defined in $\R^2$ as a function of the Riemannian distance by looking at   $r$ as function of $d(0,x)$ through Equation~\eqref{eq:riemannian_distance}. Although, in general, it is not possible to get a closed form for $r$ as a function of the Riemannian distance, in the cases we are interested in, corresponding to  constant curvature surfaces, this is possible and will be the starting point of the next section.


\subsection{Central force problem on constant curvature surfaces}

We start by considering the configuration space  $(\R^2, g)$, equipped with polar coordinates $(r, \vartheta)$, where $g$ is a conformally flat metric and we denote by $\mathbb S^2_R$  (resp. $\mathbb H^2_R$) the sphere (resp. the pseudo-sphere) of radius $R$. The conformal factor is given by 
\[
\mu_R(r):= \begin{cases}
	\dfrac{2R^2}{R^2 + r^2} & \textrm{ for  $\mathbb S^2_R$ }\\[7pt]
	\dfrac{2R^2}{R^2 - r^2} & \textrm{ for  $\mathbb H^2_R$  } \\[5pt]
\end{cases}
\] 
In terms of the curvature $\kappa$, the conformal factor can be written at once as
\[
\mu_R(r)= \dfrac{2}{1+\kappa r^2} \quad \textrm{ where } \quad \kappa= \begin{cases}
 	1/R^2 & \textrm{ for the sphere}\\
 	-1/R^2 & \textrm{ for the pseudosphere}
 \end{cases}
\]
We now take the origin as the center of the central force and we consider the simple mechanical system $(M, g, V)$ where $M:=\R^2\setminus \{(0,0)\}$ and $V:M\to \R$ is a power law potential energy (independent on $\vartheta$) depending only on the Riemannian distance from the origin. By taking into account  Proposition~\ref{thm:fava2} and by a direct integration of the conformal factor, we get that the distance of the point $P(r, \vartheta)$ to the origin is 
\[
d_R(r)=\begin{cases}
	2 R \arctan(r/R) & \textrm{ for } \quad \mathbb S^2_R\\[7pt]
	R \ln\left(\dfrac{R+r}{R-r}\right)& \textrm{ for }\quad \mathbb H^2_R
	\end{cases}
\]
Given  $\alpha \in \R^*$, we let 
\[
V_\alpha: M \to \R \quad \textrm{ defined by } \quad V_\alpha(r, \vartheta)= - m\, [d_R(r)]^\alpha \qquad m \in (0,+\infty)
\]
and we consider the Lagrangian $\widetilde L_\alpha$ of the mechanical system $(M,g,V_\alpha)$ on the state space $TM$ given by 
\[
\widetilde  L_\alpha(r,\vartheta, v_r, v_\vartheta)=\dfrac12 \mu_R^2(r)(v_r^2+ r^2 v_\vartheta^2) - V_\alpha (r, \vartheta)
\]
By setting $\xi:=r/R$, then 
the Lagrangian function can be written,  as follows
 \[
\widetilde  L_\alpha(\xi,\vartheta, v_\xi, v_\vartheta):= \begin{cases}
 	\dfrac{ 2 R^2}{(1+\xi^2)^2}\Big[v_\xi^2+ \xi^2  v_\vartheta^2\Big]+   m\big[2R\arctan \xi]^\alpha & \textrm{ for } \mathbb S^2_R\\[12pt]
 	\dfrac{ 2R^2}{(1-\xi^2)^2}\Big[v_\xi^2+ \xi^2  v_\vartheta^2\Big]
 	+m\left[R\ln\left(\dfrac{1+\xi}{1-\xi}\right)\right]^\alpha& \textrm{ for } \mathbb H^2_R.
 \end{cases}
\] 
Now, computing $\widetilde L_\alpha$ along a smooth curve and rescaling time by setting 
\[
t :=\begin{cases}
	\big(m\, 2^{\alpha{-1}}\,R^{\alpha-2}\big)^{-1/2} \tau & \textrm{ in the case of } \mathbb S^2_R \\[12pt]
	\left(m\,{2^{-1}}\, R^{\alpha-2}\right)^{-1/2}\tau & \textrm{ in the case of } \mathbb H^2_R
\end{cases}
\]
by abusing notation and  denoting  by $\cdot $ the $\tau$ derivative as well, we get that $ \widetilde L_\alpha= C_\alpha\,L_\alpha$ where 
\[
 L_\alpha(\xi,\vartheta,  \dot \xi, \dot \vartheta):= \begin{cases}
 	\dfrac{ 1}{(1+\xi^2)^2}\Big[\dot \xi^2+ \xi^2 \dot  \vartheta^2\Big]+   \arctan^\alpha\xi & \textrm{ in the case of } \mathbb S^2_R\\[12pt]
 	\dfrac{ 1}{(1-\xi^2)^2}\Big[\dot \xi^2+ \xi^2  \dot \vartheta^2\Big]+    \ln^\alpha\left(\dfrac{1+\xi}{1-\xi}\right)& \textrm{ in the case of } \mathbb H^2_R
 \end{cases}
\] 
and $C_\alpha:=m R^\alpha 2^\alpha$ (resp. $C:=m R^\alpha$) in the case of  $\mathbb S_R^2$ (resp. $\mathbb H_R^2$).


\subsection{Euler-Lagrange equation and Sturm-Liouville problem}

We let $p_\pm(\xi):= 2(1\pm\xi^2)^{-2}$,  $q_+(\xi):= \arctan^\alpha \xi$ and
 $q_-(\xi):=\ln^\alpha\left(\frac{1+\xi}{1-\xi}\right)$.
 \begin{note}
In shorthand notation, with abuse of notation, we  use $p(\xi)$ for denoting either $p_+(\xi)$ or $p_-(\xi)$ and  $q(\xi)$ for denoting either  $q_+(\xi)$ or $q_-(\xi)$. Furthermore, we denote by $p'$ (resp. $q'$) the $\xi$ derivative of $p$ (resp. $q$).
\end{note}
Bearing this notation in mind and once observed that  without loss of generality we can consider the Lagrangian $L_\alpha$ instead of $\widetilde L_\alpha$, given by 
\begin{equation}
L_\alpha(\xi, \vartheta,\dot \xi,\dot \vartheta)=\dfrac12 p(\xi)\big[\dot \xi^2+\xi^2\dot \vartheta^2\big]+q(\xi).
\end{equation} 
and by a direct calculation we get that  the associated  Euler-Lagrangian equation is 
\begin{equation}\label{eq:e-l-equation-general}
\begin{cases}
	\dfrac{d}{d\tau}(p\,\dot\xi)=\dfrac12 p'(\dot \xi^2+\xi^2\dot \vartheta^2)+p\, \xi\dot \vartheta^2+\, q'\qquad \textrm{ on } [0,T]\\[10pt]
\dfrac{d}{d\tau}(p\,\xi^2\dot \vartheta)=0.
\end{cases}
\end{equation}
A special class of solutions of the Euler-Lagrange equation, is provided by the circular solutions pointwise defined by $\gamma_0(t)=\big(\xi_0, \vartheta(t)\big)$, where by the first equation of \eqref{eq:e-l-equation-general}, we immediately get
\begin{equation}\label{eq:dot-theta-squar}
\dot \vartheta^2=\dfrac{-2q'}{(p\,\xi^2)'}\Big\vert_{\xi=\xi_0}.
\end{equation}
\begin{note}\label{notation}
	We set 
	\begin{align}\label{eq:simplfied-notations-p-q}
&p_0\=p(\xi_0),&& p'_0\= p'(\xi_0),&& p''_0\=p''(\xi_0)\\
&q_0\=q(\xi_0),&& q'_0\= q'(\xi_0),&& q''_0\=q''(\xi_0)\\
&\eta_0\=p_0\,\xi_0^2 && \eta_0'\=(p_0\,\xi_0^2)'=p'_0\xi_0^2+ 2p_0\,\xi_0 && \dot\vartheta_0^2=-2q_0'\,{\eta'}_0^{-1}.
\end{align}
\end{note}
So, the period $T$ is given by 
\begin{equation}\label{eq:T-star}
T=2\pi \,\omega_0\quad  \textrm{ where } \quad \omega_0:=\left.\sqrt{\dfrac{\eta'}{-2q'}}\right\vert_{\xi=\xi_0}
\end{equation}
By linearizing along $\gamma_0$ we get the Sturm-Liouville equation given by 
\[
-\dfrac{d}{d\tau}\big(P \dot y + Q y\big)+ \trasp{Q}\dot y + R y=0 \quad \textrm{ on } \quad [0,T]
\]
where 
\begin{multline}
P=\begin{bmatrix}
p_0&0\\0&\eta_0
\end{bmatrix},
 \qquad Q=\begin{bmatrix}
0&0\\\zeta_0&0
\end{bmatrix} \quad \textrm{ where  } 
\quad \zeta_0:=\begin{cases}
\sqrt{-2q_0'\,\eta_0'},&\quad \textrm{ if }\quad \eta'_0\ge  0\\[7pt]
-\sqrt{-2q_0'\,\eta_0'},&\quad \textrm{ if }\quad \eta'_0<0
\end{cases} \textrm{ and finally}\\[7pt]
R=\begin{bmatrix}
R_{11}&0\\0&0
\end{bmatrix}\quad \textrm{ for  } \quad R_{11}:=\eta_0' \cdot \left[\dfrac{q_0'}{\eta_0'}\right]'
\end{multline}
We set  $B=\begin{bmatrix}
P^{-1}&-P^{-1}Q\\-\trasp{Q}P^{-1}&\trasp{Q}P^{-1}Q-R
\end{bmatrix}$. By a direct computation, we get 
\begin{multline}
	P^{-1}=\begin{bmatrix}
p^{-1}_0&0\\0&\eta_0^{-1}
\end{bmatrix} \qquad P^{-1}Q=\begin{bmatrix}
0&0\\\eta_0^{-1}\cdot \zeta_0&0
\end{bmatrix}\\[10pt]
\trasp{Q}P^{-1}Q=\begin{bmatrix}
\eta_0^{-1}\cdot \zeta_0^2&0\\0&0
\end{bmatrix}\qquad \trasp{Q}P^{-1}Q-R=\begin{bmatrix}
\eta_0^{-1}\cdot \zeta_0^2-\eta_0'\cdot(q_0'/\eta_0')'&0\\0&0
\end{bmatrix}.
\end{multline}
We observe that 
\begin{multline}
\eta_0^{-1}\,\zeta_0^2= -2\,q_0'\,\big(\ln\eta_0\big)'\qquad  
\eta_0^{-1}\,\zeta_0= \sqrt{2\, q_0'(\eta_0^{-1})'}\qquad \textrm{ and }\\
\eta_0^{-1}\,\zeta_0^2- R_{11}=-2q_0'\,(\ln \eta_0)'-\eta_0'\left(\dfrac{q_0'}{\eta_0'}\right)'
\end{multline}
In conclusion, we get 
\begin{multline}
B=\begin{bmatrix}
p_0^{-1}&0&0&0\\0&\eta_0^{-1}&-\eta_0^{-1}\cdot \zeta_0&0\\ 0&-\eta_0^{-1}\cdot \zeta_0&-2q_0'\,(\ln \eta_0)'-\eta_0'\left(\dfrac{q_0'}{\eta_0'}\right)'&0\\0&0&0&0
\end{bmatrix} \textrm{ and }\\
JB=
\begin{bmatrix}
0&\eta_0^{-1}\cdot \zeta_0&2q_0'\,(\ln \eta_0)'+\eta_0'\left(\dfrac{q_0'}{\eta_0'}\right)'&0\\
0&0&0&0\\
p_0^{-1}&0&0&0\\
0&\eta_0^{-1}&-\eta_0^{-1}\cdot \zeta_0&0
\end{bmatrix}
\end{multline}
where $J:=\begin{bmatrix} 0 & -\Id_2\\ \Id_2 & 0\end{bmatrix}$ denotes the standard complex structure.


\section{The generalized Conley-Zehnder index  and normal forms}\label{sec:general-index}

This section is devoted to provide a complete proof of the following result. 

\begin{thm}\label{thm:main1}
Let  $a, b, c, d \in \R$, $a$ positive, we let  
\[
A=\begin{bmatrix}
0&b&d&0\\0&0&0&0\\a&0&0&0\\0&c&-b&0
\end{bmatrix}
\]
 and we denote by $\gamma$  the fundamental solution of the following linear Hamiltonian system 
\begin{equation}\label{eq:linear-system-general-intro}
\dot z(t)=A z(t)\qquad  t\in [0,T].
\end{equation}
Then the generalized Conley-Zehnder index $\iota_1(\gamma(t), t\in[0,T])$	 is given by
\begin{equation}
\iota_1(\gamma(t), t\in[0,T])=\begin{cases}
2k &\textrm{ if }\quad d<0 \textrm{ and }cd+b^2\ge  0\\[10pt]
2k+1 &\textrm{ if }\quad d<0 \textrm{ and } cd+b^2< 0\\[10pt]
0 &\textrm{ if }\quad \begin{cases}
d=0, \ b=0 \textrm{ and }c>0\\
d=0 \textrm{ and } b\neq 0\\ 
d>0 \textrm{ and } cd+b^2>0
\end{cases}\\[20pt]
-1 &\textrm{if }\quad \begin{cases}
d=0, \ b=0 \textrm{ and }c\le0\\ 
d>0 \textrm{ and } cd+b^2\le 0
\end{cases}
\end{cases}
\end{equation}
where $k\in\N$ is given by $k\cdot 2\pi<\sqrt{-ad}\cdot T\le (k+1)\cdot 2\pi$.
\end{thm}
\begin{rem}
We observe that in the cases we are interested in $c>0$. A direct consequence of this fact, as expected, is that $\iota_1(\gamma) \ge 0$. In fact, the last case can be easily excluded by observing that, since $c>0$ and assuming $d>0$ as required in the last case, the term $cd+b^2>0$.  
\end{rem}
Proof of Theorem~\ref{thm:main1}, will be done through  an explicit computation of the normal forms and of  the generalized Conley-Zehnder index for autonomous linear Hamiltonian  vector field of the form $JB$. We refer to the Appendix~\ref{sec:preliminary}  and references therein, for the basic definitions and properties of the generalized Conley-Zehnder index $\iota_1$, as well as the generalized Maslov index $\iCLM$. 

\begin{note}
	We let 
\begin{equation}\label{eq:notations-a-b-c-d}
a\=p_0^{-1},\quad  b\=\eta_0^{-1}\cdot \zeta_0,\quad  c\=\eta_0^{-1},\quad  d\=2q_0'\,(\ln \eta_0)'+\eta_0'\left(\dfrac{q_0'}{\eta_0'}\right)'.
\end{equation}
\end{note}
Bearing this notation in mind the linear autonomous Hamiltonian system $z'(t)=JBz(t)$ reads as 
\begin{equation}\label{eq:linear-system-general}
\dot z(t)=A z(t) \qquad t\in [0,T]
\end{equation}
where 
\begin{equation}\label{eq:matrix-A-general}
A=\begin{bmatrix}
0&b&d&0\\0&0&0&0\\a&0&0&0\\0&c&-b&0
\end{bmatrix}.
\end{equation}
By a direct computation of the Cayley-Hamilton polynomial, we get that 
\[
\det(A-\lambda \Id)=\lambda^2(\lambda^2-ad)=0\quad \Rightarrow \lambda=0\quad \textrm{ or }\quad  \lambda^2=ad
\]
So, we have to split our discussion according to the sign of $d$.

\subsection*{First case: $d$ negative} Since, we have $ad<0$, the eigenvalues of matrix $A$ are 
\begin{equation}
\lambda_1=\lambda_2=0,\qquad \lambda_3=\sqrt{-ad}\ i,\qquad \lambda_4=-\sqrt{-ad}\ i.
\end{equation} 
Since $0$ could have a non-trivial Jordan blocks, we discuss these two cases accordingly. 
\begin{itemize}
	\item {\bf Non-trivial Jordan block}. This corresponds to the case  $b^2+cd\neq 0$ since if this holds, then the  $0$ eigenvalue is not semisimple.
	
	In this case,  by the discussions in \cite[pag. 118-122]{CLW94}, there exists a symplectic matrix $S\in\Sp(4,\R)$ such that
	\begin{equation}
	A=S\Big(\begin{bmatrix}0&s_0\\0&0\end{bmatrix}\diamond \begin{bmatrix}
	0&-\sqrt{-ad}\\ \sqrt{-ad}&0
	\end{bmatrix}\Big)S^{-1},
	\end{equation}
	where $s_0$ is a constant that will be determined later. In this case the fundamental solution $\gamma$ of the system given in Equation~\eqref{eq:linear-system-general} splits as follows
	\begin{equation}\label{eq:splitting-of fundamental-solution-d<0}
	\gamma(t)=S(\gamma_1(t)\diamond \gamma_2(t))S^{-1},
	\end{equation}
	where $\gamma_1(t)\=\begin{bmatrix}1&s_0t\\0&1
	\end{bmatrix}$ and $\gamma_2(t)\=R(\sqrt{-ad}t)$ denotes  a rotation matrix. By  Lemma~\ref{lem:property of maslov-type index}, in order to compute $\iota_1(\gamma)$ we only need to compute $\iota_1(\gamma_1)$ and $\iota_1(\gamma_2)$ respectively.
	
	By Lemma \ref{lem:index of three special paths}, we have 
	\[
	\iota_1(\gamma_2)=1+2k \textrm{ where $k$ is such that } k\cdot 2\pi<\sqrt{-ad}\,T\le (k+1)\cdot 2\pi.
	\]
	 To compute the index $\iota_1(\gamma_1)$, we only need to determine the sign of $s_0$. We start by observing that $e_4=\trasp{(0,0,0,1)}$ belongs to the kernel of  $A$.  Now, we assume that $v$ is the solution of $Av= e_4$ and that  $\omega(e_4,v)=1$. Then we get that,
	\[
	v=\trasp{(0,-1,b/d, x_4)} \qquad \textrm{ for } x_4 \in \R.
	\]
We let $S^{-1}=[s_{m,n}]_{m,n=1}^4$ and $M$ be the phase flow map at time-$1$, namely 
\[
M=\begin{bmatrix}
	0&0&s_0&0\\ 0&0&0&-\sqrt{-ad}\\0&0&0&0\\0&\sqrt{-ad}&0&0
	\end{bmatrix}.
\] 
Since $A e_4=0$ and $S$ is invertible, then $MS^{-1}e_4=0$. By direct computations we have  $S^{-1}e_4=\trasp{(\alpha,0,0,0)}$ and $\alpha\neq0$. Let $S^{-1}v=\trasp{(w_1,w_2,w_3,w_4)}$. Using the equation  $Av =\xi e_4$, we get $S^{-1}v=\trasp{(v_1,0,\beta\alpha/s_0, 0)}$ where $\beta:=-(b^2+cd)/d$. Since
	$
	1=\omega(e_4,v)=\omega(S^{-1}e_4,S^{-1}v),
	$
	then we conclude that  $s_0=\beta \alpha^2$ and consequently the sign of $s_0$ is the same as the sign of $\beta$. 
	
	Recall that   $d<0$, then the sign of $s_0$ is the same as $cd+b^2$. Thus
	\begin{equation}
	\iota_1(\gamma_1)=\begin{cases}
	-1\quad &\textrm{ if } \quad cd+b^2>0\\
	0\quad &\textrm{ if } \quad cd+b^2<0.
	\end{cases}
	\end{equation}
	Consequently, by the additivity property of the $\iota_1$ index w.r.t. the symplectic product, we have
	\begin{equation}\label{eq:index-d-negative-and-zero-not-semisimple}
	\iota_1(\gamma)=	\iota_1(\gamma_1)+\iota_1(\gamma_2)=\begin{cases}
	2k,\quad &\text{if}\quad cd+b^2>0\\
	1+2k,\quad &\text{if}\quad cd+b^2<0\\
	\end{cases}
	\end{equation} 
	where $k$ is determined by
	\begin{equation}\label{eq:k-determined}
	k\cdot 2\pi<\sqrt{-ad}\cdot T\le (k+1)\cdot 2\pi.
	\end{equation}  	 
	
	\item {\bf Trivial Jordan blocks.} This corresponds to the case $b^2+cd= 0$, since if this holds, then eigenvalue $0$ is semisimple.	 
	
	 In this case,  by the discussions in \cite[pag. 118-122]{CLW94}, there exists a symplectic matrix $S\in\Sp(4,\R)$ such that
	\begin{equation}
	A=S\Big(\begin{bmatrix}0&0\\0&0\end{bmatrix}\diamond \begin{bmatrix}
	0&-\sqrt{-ad}\\ \sqrt{-ad}&0
	\end{bmatrix}\Big)S^{-1}.
	\end{equation}
	
	Let $\gamma(t)$ be the fundamental solution of system \eqref{eq:linear-system-general}, then we have
	\begin{equation}\label{eq:splitting-of fundamental-solution-d<0-simple}
	\gamma(t)=S(\gamma_1(t)\diamond \gamma_2(t))S^{-1},
	\end{equation}
	where $\gamma_1(t)=I$ and $\gamma_2(t)=R(\sqrt{-ad}t)$. 
	
	By Lemma \ref{lem:index of three special paths} we have 
	\[
	\iota_1(\gamma_1)=-1, \iota_1(\gamma_2)=1+2k \textrm{ where $k$ is given by } k\cdot 2\pi<\sqrt{-ad}\,T\le (k+1)\cdot 2\pi.
	\]
	 Then by Lemma \ref{lem:property of maslov-type index} we have
	\begin{equation}\label{eq:index-d-negative-and-zero-semisimple}
	\iota_1(\gamma)=\iota_1(\gamma_1)+\iota_1(\gamma_2)=2k,
	\end{equation}
	where $k$ is obtained by $k\cdot 2\pi<\sqrt{-ad}\,T\le (k+1)\cdot 2\pi$. 
	
\end{itemize}

\subsection*{Second case: $d=0$}

In this case the spectrum of  $A$ is given by  $0$ and $A$ is not semisimple. By direct computations we have
\begin{multline}
A^2=\begin{bmatrix}
0&0&0&0\\0&0&0&0\\0&ab&0&0\\-ab&0&0&0
\end{bmatrix} \qquad 
A^3=\begin{bmatrix}
0&0&0&0\\0&0&0&0\\0&0&0&0\\0&-ab^2&0&0
\end{bmatrix} \textrm{ and } 
A^j=0\qquad  \forall j\ge 4.
\end{multline}
Then the fundamental solution of the system  given at Equation~\eqref{eq:linear-system-general} is given by 
\begin{equation}
\gamma(t)=e^{A t}=\begin{bmatrix}
1&bt&0&0\\0&1&0&0\\at&abt^2/2 &1&0\\-abt^2/2&ct-ab^2t^3/6&-bt&1
\end{bmatrix}.
\end{equation}
We let  
\[
M_{1,1}(t)=\begin{bmatrix}
1&bt\\0&1
\end{bmatrix}\qquad M_{2,1}(t)=\begin{bmatrix}
at&abt^2/2 \\-abt^2/2&ct-ab^2t^3/6
\end{bmatrix} \quad \textrm{  and } \quad M_{22}(t)=\begin{bmatrix}
1&0\\-bt&1
\end{bmatrix}.
\]
In this case we can rewrite $\gamma(t)$ as a two by two block matrix. The computation of  the $\iota_1$ index  can be done, for instance, by using Equation~\eqref{eq:zhu's formula to compute maslov index}. Bearing in mind the notation given  in Equation~\ref{eq:app1} and formula given  in Equation~\eqref{eq:zhu's formula to compute maslov index}, we let $V=\Gr(I_2)$.  Then $V^I=\Gr(I_2)$ and $W_I(V)=\Gr(I_4)$. Since $M_{11}(0)=I_2$, then we have $S(0)=\R^2$. We need to compute  $S(T)$ which is related to the coefficient $b$. For this reason, we have to split into two subcases. 
\begin{itemize}
	\item[{\bf ($b=0$)}] In this case $M_{11}(T)=I_2$ and consequently $S(T)=\R^2$. Being 	$M_{21}(0)=0$, then we have 
	\begin{equation}
	\trasp{M}_{11}M_{21}=0\qquad \Rightarrow \qquad m^+(\trasp{M}_{11}M_{21}|_{S(0)})=0
	\end{equation}
	Moreover, we have $M_{21}(T)=\begin{bmatrix}
	aT&0\\0&cT
	\end{bmatrix}$ and so 
	\begin{equation}
	m^+(\trasp{M}_{11}(T)M_{21}(T)|_{S(T)})=\begin{cases}
	1\quad &\textrm{ if }\quad c\le 0\\
	2\quad &\textrm{ if }\quad c> 0.
	\end{cases}
	\end{equation}
	So, by Formula \eqref{eq:zhu's formula to compute maslov index}, we get
	\begin{equation}
	\iCLM(\Gr(I_4),\Gr(\gamma(t))=\begin{cases}
	1-0+2-2=1,\quad &\textrm{ if}\quad c\le  0\\
	2-0+2-2=2,\quad &\textrm{ if}\quad c> 0\\
	\end{cases}
	\end{equation}
	and by invoking Lemma~\ref{lem:relation between two indices}, then we conclude that  
	\begin{equation}\label{eq:index-d-zero-b-zero}
	\iota_1(\gamma(t))=\begin{cases}
	-1\quad &\textrm{ if }\quad c\le 0\\
	0 \quad &\textrm{ if }\quad c> 0
	\end{cases}
	\end{equation}	
	\item[{\bf ($b\neq 0$)}] In this case $M_{11}(T)=\begin{bmatrix}
	1&bT\\0&1
	\end{bmatrix}$ and consequently $S(T)=\{ \trasp{(x,0)}\ | \ x\in\R  \}$.		
	 Thus, 
	\begin{equation}
	\trasp{M}_{11}(T)M_{21}(T)=\begin{bmatrix}
	1&0\\bT&1
	\end{bmatrix}\cdot \begin{bmatrix}
	aT&abT^2/2\\-abT^2/2&cT-ab^2T^3/6
	\end{bmatrix}=\begin{bmatrix}
	aT&abT^2/2\\ \frac{ab}{2}T^2&cT+ab^2T^3/3
	\end{bmatrix}.
	\end{equation}
	Then, for every $X=\trasp{(x,0)}\in S(T)$ we have
	\begin{equation}
	\Mprod{\trasp{M}_{11}(T)M_{21}(T)X}{X}=\Mprod{\begin{bmatrix}
		aT&abT^2/2\\ abT^2/2&cT+ab^2T^3/3
		\end{bmatrix}\begin{bmatrix}
		x\\0
		\end{bmatrix}}{\begin{bmatrix}
		x\\0
		\end{bmatrix}}=aTx^2>0.
	\end{equation}	
	Therefore, by \eqref{eq:zhu's formula to compute maslov index} we have
	\begin{equation}
	\iCLM(Gr(I_4),Gr(\gamma(t))=1-0+2-1=2.
	\end{equation}	
	By Lemma~\ref{lem:relation between two indices}	 once again, we have
	\begin{equation}\label{eq:index-for-d-zero-b-nonzero}
	\iota_1(\gamma)=\iCLM(Gr(I_4),Gr(\gamma(t))-2=0.
	\end{equation}
	
\end{itemize}

\subsection*{Third case: $d$ positive}

All eigenvalues of $A$ are $\lambda_1=\lambda_2=0$ and $\lambda_3=\sqrt{ad}, \lambda_4=-\sqrt{ad}$. 
Since $0$ could have a non-trivial Jordan blocks, we discuss these two cases accordingly. 
\begin{itemize}
	\item {\bf Non-trivial Jordan block}. If  $b^2+cd\neq 0$, then eigenvalue $0$ is not semisimple.
	
	In this case,  by the discussions in \cite[pag. 118-122]{CLW94}, there exists a symplectic matrix $S\in\Sp(4,\R)$ such that
	\begin{equation}
	A=S\Big(\begin{bmatrix}0&s_0\\0&0\end{bmatrix}\diamond \begin{bmatrix}
	\sqrt{ad}&0\\ 0&-\sqrt{ad}
	\end{bmatrix}\Big)S^{-1},
	\end{equation}
	where $s_0$ is a constant that will be determined later. Let $\gamma(t)$ be the fundamental solution of system given at Equation~\eqref{eq:linear-system-general}. Then we have
	\begin{equation}\label{eq:splitting-of fundamental-solution-d>0}
	\gamma(t)=S(\gamma_1(t)\diamond \gamma_2(t))S^{-1},
	\end{equation}
	where $\gamma_1(t)=\begin{bmatrix}1&s_0t\\0&1
	\end{bmatrix}$ and $\gamma_2(t)=\begin{bmatrix}
	e^{\sqrt{ad}t}&0\\0&e^{-\sqrt{ad}t}
	\end{bmatrix}$. As before, by taking into account Lemma~\ref{lem:property of maslov-type index}, in order to compute $\iota_1(\gamma)$ we only need to compute $\iota_1(\gamma_1)$ and $\iota_1(\gamma_2)$ respectively.
	
	We start observing that for the path $t \mapsto\gamma_2(t)$ the unique crossing occurs at $t=0$ and, at this crossing, $\ker(\gamma(0)-I_2)=\R^2$. Moreover, the crossing form is represented w.r.t. to the scalar product of $\R^2$ by the matrix $-J\dot \gamma(0)\gamma(0)^{-1}=\begin{bmatrix}
	0&-\sqrt{ad}\\-\sqrt{ad}&0
	\end{bmatrix}$. By a straightforward computation we get that the  eigenvalues $\pm \sqrt{ad}$ and since the contribution to the Maslov index at the starting point is provided by the co-index according to formula given at  Equation~\eqref{eq:maslov index formula}, then we get that  $\iCLM(\Delta,\Gr \gamma_2(t), t \in [0,T])=1$. Consequently, by invoking  Lemma~\eqref{lem:relation between two indices},  we finally get	$\iota_1(\gamma_2)=0$.

	For computing the index $\iota_1(\gamma_1)$, we only need to determine the sign of $s_0$. Since the discussion is precisely the same as before will be left to the reader. We just mention that, in this case, the sign of $s_0$ coincides with that of $-(cd+b^2)$. Thus, we get 
	\begin{equation}
	\iota_1(\gamma_1)=\begin{cases}
	0,\quad &\text{if}\quad cd+b^2>0\\
	-1,\quad &\text{if} \quad cd+b^2<0
	\end{cases}
	\end{equation}
	Consequently, we have
	\begin{equation}\label{eq:index-d-positive}
	\iota_1(\gamma)=	\iota_1(\gamma_1)+\iota_1(\gamma_2)=\begin{cases}
	0,\quad &\text{if}\quad cd+b^2>0\\
	-1,\quad &\text{if} \quad cd+b^2<0
	\end{cases}
	\end{equation} 
	
	\item {\bf Trivial Jordan block.} If  $b^2+cd= 0$, then the  $0$ eigenvalue is semisimple.	 
	
	In this case,  by the normal form classification given at \cite[page 118-122]{CLW94}, there exists a symplectic matrix $S\in\Sp(4,\R)$ such that
	\begin{equation}
	A=S\Big(\begin{bmatrix}0&0\\0&0\end{bmatrix}\diamond \begin{bmatrix}
	\sqrt{ad}&0\\ 0&-\sqrt{ad}
	\end{bmatrix}\Big)S^{-1}.
	\end{equation}
	By the very same computation as before, we get 
	$\iota_1(\gamma_1)=-1$ and  $\iota_1(\gamma_2)=0$. Then by using once again Lemma~\ref{lem:property of maslov-type index}, then we get  
	\begin{equation}\label{eq:index-d-positive-trivial-jordan}
	\iota_1(\gamma)=\iota_1(\gamma_1)+\iota_1(\gamma_2)=-1.
	\end{equation}
	
\end{itemize}

\subsection*{Proof of Theorem~\ref{thm:main1}}
The proof of this result readily follows by summing up 
Equation~\eqref{eq:index-d-negative-and-zero-not-semisimple}, Equation~\eqref{eq:index-d-negative-and-zero-semisimple},
Equation~\eqref{eq:index-d-zero-b-zero}, Equation~\eqref{eq:index-for-d-zero-b-nonzero}, Equation~\eqref{eq:index-d-positive} and finally Equation~\eqref{eq:index-d-positive-trivial-jordan}, concluding the proof.
\subsection*{Proof of Theorem~\ref{thm:main1-intro}} The proof of this result is direct consequence of Theorem~\ref{thm:main1} and Theorem~\ref{lem:relation between morse and maslov-type}.


\section{Computations of Maslov index for circular orbits}\label{sec:computations of maslov index}

This section  is devoted to compute the generalized Conley-Zehnder index of the circular orbits on the sphere, pseudo-sphere and Euclidean plane. This will be achieved by using Theorem~\ref{thm:main1} and the computation provides in Section~\ref{sec:general-index}. So, the main issues is to explicitly compute the constants $a,b,c,d$ appearing in the aforementioned theorem.

We start by recalling that, in this case, the Lagrangian of the problem  is given by 
\begin{equation}
L(\xi, \vartheta,\dot \xi,\dot \vartheta)=\dfrac12 p(\xi)\big[\dot \xi^2+\xi^2\dot \vartheta^2\big]+q(\xi).
\end{equation}


\subsection{Circular orbits on the sphere}
In this case $
p(\xi)=\dfrac{2}{(1+\xi^2)^2} \textrm{ and } q(\xi)= \arctan^\alpha \xi$. In shorthand notation, we set  $F(\xi):=\arctan\xi$. 
Bearing in mind the notation given  in Equation~\eqref{eq:simplfied-notations-p-q}, we get
\begin{equation}\label{eq:derivatives-p-and-q-sphere}
\begin{aligned}
&p_0= 2(1+\xi_0^2)^{-2}  & & p'_0=-8\xi_0(1+\xi_0^2)^{-3} & &
p''_0=8(5\xi^2_0-1)(1+\xi_0^2)^{-4} \\
&q_0= F^\alpha(\xi_0)   &&q'_0=\alpha\cdot F^{\alpha-1}(\xi_0)(1+\xi^2_0)^{-1} && 
q''_0=\dfrac{\alpha(\alpha-1)F^{\alpha-2}(\xi_0)-2\alpha\xi_0 F^{\alpha-1}(\xi_0)}{(1+\xi^2_0)^{2}}\\
&\eta_0=2\xi_0^2(1+\xi_0^2)^{-2} && \eta_0'=4\xi_0(1-\xi_0^2)(1+\xi_0^2)^{-3} && \dot \vartheta_0^2=\dfrac{-\alpha F^{\alpha-1}(\xi_0)(1+\xi_0^2)^2}{2\xi_0(1-\xi_0^2)}
\end{aligned}
\end{equation}
In order for the (RHS) of $\dot \vartheta_0^2$ to be positive, we have to impose the following restriction on $\xi_0$: 
\begin{equation}\label{eq:range-of-xi-sphere}
\xi_0 \in \begin{cases}
 (0,1)\quad &\text{if}\  \alpha<0,\\ (1,+\infty)\quad &\text{if}\ \alpha>0. 
\end{cases}
\end{equation}
By a direct computation, we get
\begin{equation}\label{eq:a-b-c-d-sphere}
\begin{aligned}
& a=\dfrac{(1+\xi_0^2)^2}{2} && 
b=\dfrac{-2\xi_0^2+2}{\xi_0}\cdot \sqrt{\dfrac{-\alpha F^{\alpha-1}(\xi_0)}{2\xi_0(1-\xi_0^2)}}\\
&c=\dfrac{(1+\xi_0^2)^2}{2\xi_0^2} && 
d=\dfrac{\alpha F^{\alpha-2}(\xi_0)\{(3\xi_0^4-2\xi_0^2+3)\cdot F(\xi_0)+(\alpha-1)\xi_0(1-\xi_0^2)\}}{\xi_0(1+\xi_0^2)^2(1-\xi_0^2)}.
\end{aligned}
\end{equation}
In order to determine the generalized Conley-Zehnder index, we have to discuss the sign of $d$ according to $\alpha$ and $\xi_0$.

\subsection*{First case:  $\alpha$ positive}

Since in this case $\xi_0\in (1,+\infty)$, then we get that the sign of $d$ is minus the sign of 
\begin{equation}\label{eq:f-1-function}
f_1(\xi)=(3\xi^4-2\xi^2+3)\cdot \arctan \xi+(\alpha-1)\xi(1-\xi^2).
\end{equation}
We let
\begin{equation}
\Omega_1\=(1,+\infty)\times (0,+\infty)
\end{equation}
and we define the subregions
\begin{equation}
\Omega_1^\pm\=\Set{(\xi,\alpha)\in\Omega|\pm f_1(\xi)>0}
\end{equation} 
whose separatrix is given by 
\begin{equation}
\Omega_1^0=\Set{(\xi,\alpha)\in\Omega_1|\alpha=1+\dfrac{(3\xi^4-2\xi^2+3)\arctan\xi}{\xi(\xi^2-1)}}.
\end{equation}
We observe that the function $h$, defined by 
\[
h(\xi):=1+\dfrac{(3\xi^4-2\xi^2+3)\arctan\xi}{\xi(\xi^2-1)}
\]
is convex, bounded below by  constant $\alpha^*>1$ and such that 
\[
\lim_{\xi \to 1^-} h(\xi)=+\infty \qquad \textrm{ and }\qquad  \lim_{\xi \to +\infty} h(\xi)= +\infty.
\]
Moreover $h$ has an oblique asymptote defined by the equation
\[
\alpha= \xi + 1
\] 
\paragraph{Case: $(\xi_0,\alpha)\in\Omega_1^-$.} Then we have $f_1(\xi_0)<0$ and consequently by Equation~\eqref{eq:a-b-c-d-sphere}, we get that $d>0$. Since $c$ is always strictly positive, then we have $b^2+cd>0$. By taking into account Equation~\eqref{eq:index-d-positive}, we have 
\begin{equation}
\iota_1(\gamma)=0\qquad \text{if}\quad (\xi,\alpha)\in\Omega_1^-.
\end{equation}
\paragraph{Case: $(\xi_0,\alpha)\in\Omega_1^0$.}  Then we have $f_1(\xi_0)=0$ and consequently by using Equation~\eqref{eq:a-b-c-d-sphere} we have $d=0$. Since $c>0$, then by Equation~\eqref{eq:index-d-zero-b-zero} and  Equation~\eqref{eq:index-for-d-zero-b-nonzero}, we also have 
\begin{equation}
\iota_1(\gamma)=0\qquad \text{if}\quad (\xi,\alpha)\in\Omega_1^0.
\end{equation}
\paragraph{Case: $(\xi_0,\alpha)\in\Omega_1^+$.} Then we have $f_1(\xi_0)>0$ and consequently by Equation~\eqref{eq:a-b-c-d-sphere} we have $d<0$. Next we need to determine the sign of $b^2+cd$. Taking into account Equation~\eqref{eq:a-b-c-d-sphere} and after some algebraic manipulations, we  have
\begin{equation}\label{eq:b-square-plus-cd-sphere}
\begin{aligned}
b^2+cd&=
\dfrac{-\alpha F^{\alpha-2}(\xi_0)\{ (\xi_0^4-6\xi_0^2+1)\cdot F(\xi_0)-(\alpha-1)\xi_0(1-\xi_0^2)\}}{2\xi_0^3(1-\xi_0^2)}.
\end{aligned}
\end{equation}
We observe that, since  $F(\xi_0)$ is strictly positive and  $1-\xi_0^2<0$, then the sign of $b^2+cd$ coincides with that   of $(\xi_0^4-6\xi_0^2+1)\cdot F(\xi_0)-(\alpha-1)\xi_0(1-\xi_0^2)$. Let 
\begin{equation}\label{eq:f-2-function}
f_2(\xi)=(\xi^4-6\xi^2+1)\arctan\xi-(\alpha-1)\xi(1-\xi^2).
\end{equation}
Arguing precisely as before, we can split $\Omega_1^+$ into following three subregions 
\begin{multline}\label{eq:omega-1-positive-subregions}
\Omega_1^{+,+}=\Big\{(\xi,\alpha)\in\Omega_1\ | \ \alpha>1+\dfrac{(\xi^4-6\xi^2+1)\arctan\xi}{\xi(1-\xi^2)}\Big \}\\
	\Omega_1^{+,-}=\Big\{(\xi,\alpha)\in\Omega_1\ | \ \alpha<1+\dfrac{(\xi^4-6\xi^2+1)\arctan\xi}{\xi(1-\xi^2)}\Big \} \qquad \textrm{ and }\\
\Omega_1^{+,0}=\left\{(\xi,\alpha)\in\Omega_1\ | \ \alpha=1+\dfrac{(\xi^4-6\xi^2+1)\arctan\xi}{\xi(1-\xi^2)}\right\}.
\end{multline}
By this discussion, we finally get that 
\begin{equation}
b^2+cd \textrm{ is } \begin{cases}
\textrm{positive } & \textrm{ for }\quad (\xi,\alpha)\in \Omega_1^{+,+}\\[7pt]
\textrm{negative }& \textrm{ for }\quad (\xi,\alpha)\in \Omega_1^{+,-}.
\end{cases}
\end{equation}
In this case for computing  the $\iota_1$ index, we need to  calculate the integer $k$ appearing  in Equation~\eqref{eq:index-d-negative-and-zero-not-semisimple} and defined by $k\cdot 2\pi<\sqrt{-ad}\,T\le (k+1)\cdot 2\pi$. Since   
\begin{equation}
T=\dfrac{2\pi}{\dot{\vartheta}}=2\pi\, \sqrt{\frac{2\xi_0(1-\xi_0^2)}{-\alpha F^{\alpha-1}(\xi_0)(1+\xi_0^2)^2}}
\end{equation}
and using Equation~\eqref{eq:a-b-c-d-sphere},  we get
\begin{multline}
\sqrt{-ad}\cdot T= 2\pi f_3(\xi) \qquad \textrm{ where }\qquad 
f_3(\xi):=\sqrt{\dfrac{(3\xi^4-2\xi^2+3)\cdot\arctan\xi+(\alpha-1)\xi(1-\xi^2)}{\arctan\xi \cdot (1+\xi^2)^2}}.
\end{multline}
By a direct computation, we infer that 
\begin{equation}
\begin{aligned}
k&=0 \iff  \alpha\ge  1+\dfrac{2(\xi^2-1)\arctan \xi}{\xi}.
\end{aligned}
\end{equation}
So, denoting  these regions as follows 
\begin{multline}
\Omega_{1,0}^{+,\pm}\=\Set{(\xi,\alpha)\in\Omega_1^{+,\pm} \ | \  \alpha\ge 1+\dfrac{2(\xi^2-1)\arctan \xi}{\xi}}\\[7pt]
\Omega_{1,0}^{+,0}\=\Set{(\xi,\alpha)\in\Omega_1^{+,0} \ | \  \alpha\ge 1+\dfrac{2(\xi^2-1)\arctan \xi}{\xi}}.
\end{multline}
we finally get that for any $(\xi,\alpha)\in\Omega_{1,0}^{+,\pm}\bigcup \Omega_{1,0}^{+,0}$, we have $k=0$.

We observe that for every $(\xi,\alpha)\in\Omega_1^+$, it holds that $f_3(\xi)\le 2$. In fact, by a direct computation, we get  
\begin{equation}\label{eq:f-3-less-that-2-1}
\begin{aligned}
f_3(\xi)\le 2 &\iff 
\alpha \ge \dfrac{\xi(1-\xi^2)+(\xi^4+10\xi^2+1)\cdot \arctan\xi}{\xi(1-\xi^2)}.
\end{aligned}
\end{equation}
By an elementary computation, it readily follows that the numerator of the rational function at the (RHS) of Equation~\eqref{eq:f-3-less-that-2-1} is negative in $\Omega_1$.
So, the only case to be considered corresponds to   $k=1$ or otherwise stated to 
\begin{equation}\label{eq:k-is-1-equivalent-condition}
1<f_3(\xi)\le 2 \quad \iff \quad  0< \alpha <1+\dfrac{2(\xi^2-1)\arctan \xi}{\xi}.
\end{equation}
We set
\begin{multline}
\Omega_{1,1}^{+,\pm}=\Set{(\xi,\alpha)\in\Omega_1^{+,\pm}|  0< \alpha< 1+\dfrac{2(\xi^2-1)\arctan \xi}{\xi}}\\
\Omega_{1,1}^{+,0}=\Set{ (\xi,\alpha)\in\Omega_1^{+,0}| 0< \alpha< 1+\dfrac{2(\xi^2-1)\arctan \xi}{\xi}}.
\end{multline}
Then for every $(\xi,\alpha)\in\Omega_{1,1}^{+,+}\bigcup \Omega_{1,1}^{+,0}$ we have $k=1$.
 \begin{figure}[ht]
	\centering
	\includegraphics[scale=0.10]{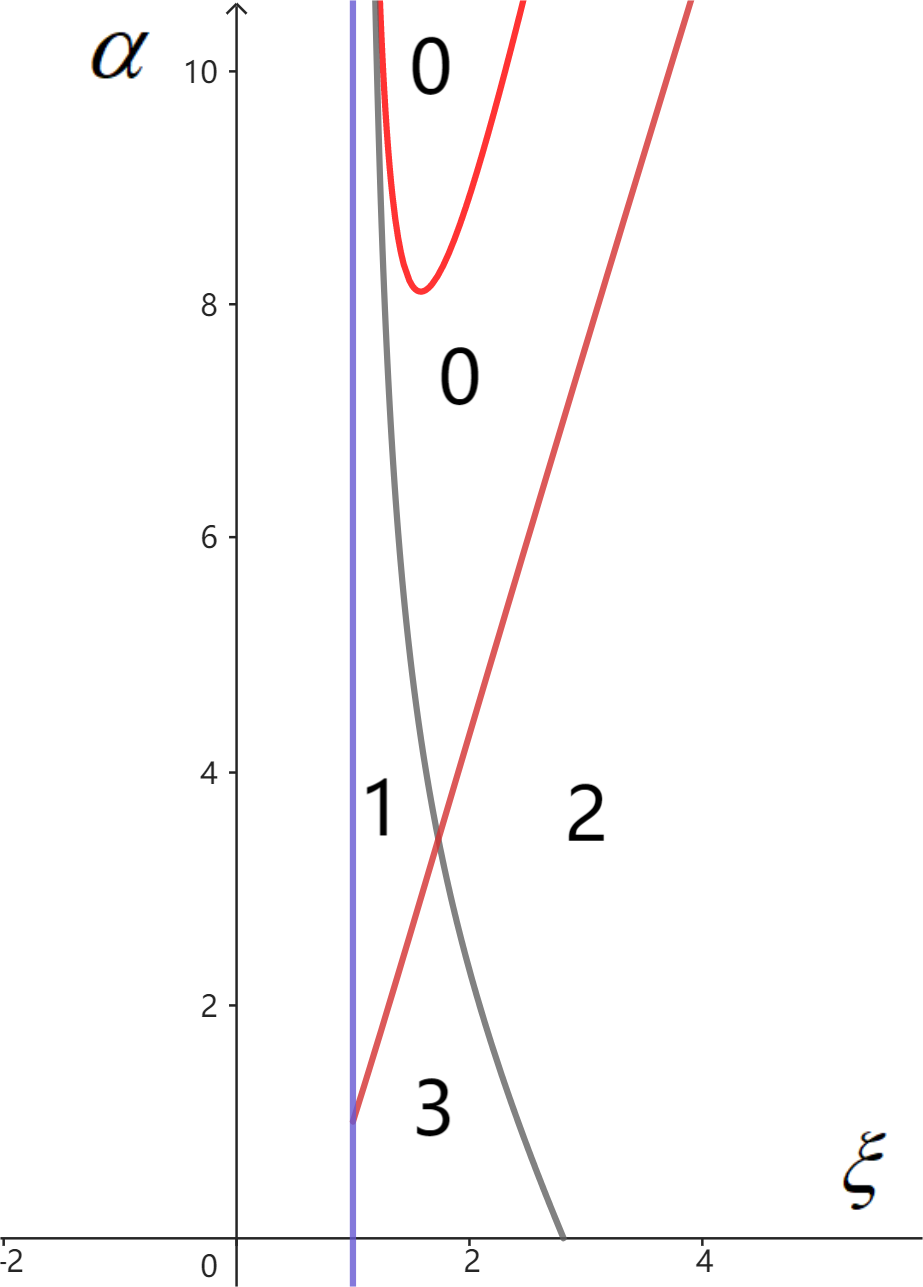}
	\caption{In the figures are displayed the subregions of the $\widehat{\xi O\alpha}$-region $\Omega_1:=(1,+\infty)\times (0, +\infty)$ labeled by the Morse index of the corresponding circular orbit} 
	\label{fig:sphere-positive-alpha}
\end{figure}

In Figure~\ref{fig:sphere-positive-alpha} are displayed all the involved regions. 

By invoking Equation~\eqref{eq:index-d-negative-and-zero-not-semisimple} and Equation~\eqref{eq:index-d-negative-and-zero-semisimple} we finally get 
\begin{equation}
\iota_1(\gamma)=\begin{cases}
0 \quad &\textrm{ if }\quad (\xi_0,\alpha)\in \Omega_1^{-}\bigcup\Omega_1^{0}\bigcup \Omega_{1,0}^{+,+}\bigcup\Omega_{1,0}^{+,0}\\
1 \quad &\textrm{ if }\quad (\xi_0,\alpha)\in \Omega_{1,0}^{+,-}\\
2\quad &\textrm{ if }\quad (\xi_0,\alpha)\in \Omega_{1,1}^{+,+}\bigcup \Omega_{1,1}^{+,0}\\
3 \quad &\textrm{ if }\quad (\xi_0,\alpha)\in \Omega_{1,1}^{+,-}\\
\end{cases}
\end{equation}


\subsection*{Second case:  $\alpha$ negative}

In this case, by taking into account the restrictions provided at Equation~\eqref{eq:range-of-xi-sphere}, we have $\xi_0\in (0,1)$ and consequently $1-\xi_0^2>0$. Arguing precisely as before, we need to establish the signs of $d$ and $b^2+cd$ as well as the value of $k$. Since all expressions  are precisely as before with  the only difference about the range of $\xi$ and $\alpha$. 

We let 
\begin{equation}
\Omega_2\=(0,1)\times (-\infty, 0).
\end{equation}
and we start observing that the sign of $f_1(\xi)$ defined in \eqref{eq:f-1-function} for $(\xi,\alpha)\in\Omega_2$ is opposite to  that of $d$. Precisely as before,  we let 
\begin{multline}
\Omega_2^\pm\=\Set{(\xi,\alpha)\in\Omega_2| \pm f_1(\xi)>0}\qquad  \textrm{ and }\\
\Omega_2^0=\Set{(\xi,\alpha)\in\Omega_2|\alpha=1+\dfrac{(3\xi^4-2\xi^2+3)\arctan\xi}{\xi(\xi^2-1)}}.
\end{multline}
\begin{figure}[ht]
	\centering
	\includegraphics[scale=0.10]{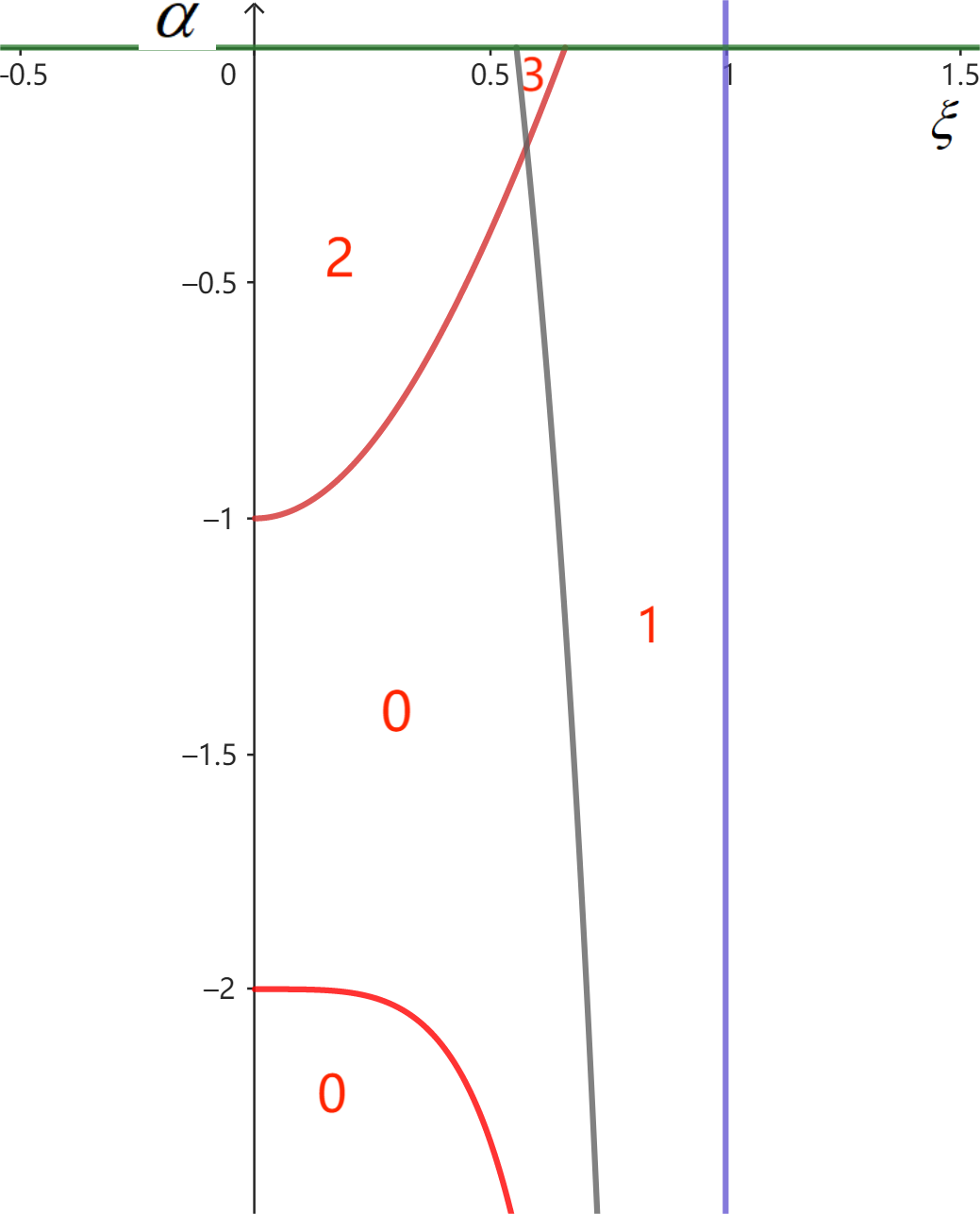}
	\caption{In the figures are displayed the subregions of the $\widehat{\xi O\alpha}$-region $\Omega_2:=(1,+\infty)\times (-\infty, 0)$ labeled by the Morse index of the corresponding circular orbit} 
	\label{fig:sphere-negative-alpha}
\end{figure}
For $(\xi,\alpha)\in\Omega_2^-$, we have $d>0$ and since $c>0$, then we get $cd+b^2>0$.  

By invoking Equation~\eqref{eq:index-d-positive}, then we get 
 \begin{equation}
\iota_1(\gamma)=0, \quad \textrm{ if }\quad (\xi,\alpha)\in\Omega_2^-.
\end{equation}
For $(\xi,\alpha)\in\Omega_2^0$, since $c>0$, and by taking into account  Equation~\eqref{eq:index-for-d-zero-b-nonzero} and Equation~\eqref{eq:index-d-zero-b-zero} we conclude that 
\begin{equation}
\iota_1(\gamma)=0, \quad \text{if}\quad (\xi,\alpha)\in\Omega_2^0.
\end{equation}
For $(\xi,\alpha)\in\Omega_2^+$, we have $d<0$. Then we need to determine the sign of $b^2+cd$. By Equation~\eqref{eq:b-square-plus-cd-sphere} we already know that the sign of $b^2+cd$ coincides with that of $f_2(\xi)$ defined at Equation~\eqref{eq:f-2-function} for $(\xi,\alpha)\in\Omega_2^+$. Then we define 
	\begin{multline}
	\Omega_2^{+,\pm}=\Set{(\xi,\alpha)\in\Omega_2^+| \pm f_2(\xi)>0} \qquad \textrm{ and }\\
	\Omega_2^{+,0}=\Set{(\xi,\alpha)\in\Omega_2^+| \ f_2(\xi)=0 }
	\end{multline}
and by an algebraic manipulation reduces to 
\begin{multline}
\Omega_2^{+,+}=\Set{(\xi,\alpha)\in\Omega_2^+|\alpha<1+\dfrac{(\xi^4-6\xi^2+1)\cdot \arctan\xi}{\xi(1-\xi^2)}}\\
\Omega_2^{+,-}=\Set{(\xi,\alpha)\in\Omega_2^+|\alpha>1+\dfrac{(\xi^4-6\xi^2+1)\cdot \arctan\xi}{\xi(1-\xi^2)}}\\
\Omega_2^{+,0}=\Set{(\xi,\alpha)\in\Omega_2^+|\alpha=1+\dfrac{(\xi^4-6\xi^2+1)\cdot \arctan\xi}{\xi(1-\xi^2)}}.
\end{multline}	
So, there holds that
\begin{equation}
cd+b^2 \textrm{ is } \begin{cases}
 	\textrm{positive} & \textrm{ for } (\xi,\alpha)\in\Omega_2^{+,+}\\
 	\textrm{negative} & \textrm{ for } (\xi,\alpha)\in\Omega_2^{+,-}\\
 	\textrm{zero} & \textrm{ for } (\xi,\alpha)\in\Omega_2^{+,0}.
 \end{cases}
\end{equation}
	Arguing precisely as before, we get that  $f_3(\xi)<2$ for every  $(\xi,\alpha)\in \Omega_2^+$. Moreover, it is straightforward to check that 
	\begin{equation}
	\begin{aligned}
	k&=0 \iff
	\alpha\le  1-\dfrac{2(1-\xi^2)\arctan \xi}{\xi}.
	\end{aligned}
	\end{equation}
	Now we define the following planar regions
	\begin{multline}
	\Omega_{2,0}^{+,\pm}\=\Set{(\xi,\alpha)\in\Omega_2^{+,\pm}| \alpha\le 1-\dfrac{2(1-\xi^2)\arctan \xi}{\xi}}\\
	\Omega_{2,1}^{+,\pm}\=\Set{(\xi,\alpha)\in\Omega_2^{+,\pm}| 1-\dfrac{2(1-\xi^2)\arctan \xi}{\xi}< \alpha< 0}\\
	\Omega_{2,0}^{+,0}\=\Set{(\xi,\alpha)\in\Omega_2^{+,0}|\alpha\le 1-\dfrac{2(1-\xi^2)\arctan \xi}{\xi}}\\
	\Omega_{2,1}^{+,0}\=\Set{(\xi,\alpha)\in\Omega_2^{+,0}| 1-\dfrac{2(1-\xi^2)\arctan \xi}{\xi}< \alpha< 0}.
	\end{multline}
	Therefore, we get that  $k=0$ for $(\xi,\alpha)\in\Omega_{2,0}^{+,\pm}\bigcup\Omega_{2,0}^{+,0}$ and $k=1$ for $(\xi,\alpha)\in\Omega_{2,1}^{+,\pm}\bigcup\Omega_{2,1}^{+,0}$. Now, by invoking Equation~\eqref{eq:index-d-negative-and-zero-not-semisimple}, we get 
	\begin{equation}
	\iota_1(\gamma)=\begin{cases}
	0 & \textrm{ if }\quad (\xi,\alpha)\in\Omega_{2,0}^{+,+}\bigcup\Omega_{2,0}^{+,0}\bigcup\Omega_2^-\bigcup\Omega_2^0\\
	1 & \textrm{ if }\quad (\xi,\alpha)\in\Omega_{2,0}^{+,-}\\
	2 & \textrm{ if }\quad (\xi,\alpha)\in\Omega_{2,1}^{+,+}\bigcup \Omega_{2,1}^{+,0}\\
	3& \textrm{ if }\quad (\xi,\alpha)\in\Omega_{2,1}^{+,-}
	\end{cases}
	\end{equation}

We finally are in position to summarize the involved discussion in the following conclusive result for the sphere. 
\begin{thm}\label{pro:index-of-sphere}
	Under above notations, the indices of a circular orbit on sphere are given by
	\begin{equation}
	\iota_1(\gamma)=\begin{cases}
	0  &\textrm{ if }\quad (\xi_0,\alpha)\in \Omega_1^{-}\bigcup\Omega_1^{0}\bigcup \Omega_{1,0}^{+,+}\bigcup\Omega_{1,0}^{+,0}\bigcup\Omega_{2,0}^{+,+}\bigcup\Omega_{2,0}^{+,0}\bigcup\Omega_2^-\bigcup\Omega_2^0\\
	1  &\textrm{ if }\quad (\xi_0,\alpha)\in\Omega_{1,0}^{+,-}\bigcup \Omega_{2,0}^{+,-} \\
	2 &\textrm{ if }\quad (\xi_0,\alpha)\in \Omega_{1,1}^{+,+}\bigcup \Omega_{1,1}^{+,0}\bigcup\Omega_{2,1}^{+,+}\bigcup \Omega_{2\bigcup,1}^{+,0} \\
	3 &\textrm{ if }\quad (\xi_0,\alpha)\in \Omega_{1,1}^{+,-}\bigcup \Omega_{2,1}^{+,-}\\
	\end{cases}
	\end{equation}
\end{thm}
A direct consequence of Proposition \ref{pro:index-of-sphere} concerns the Morse index of the circular orbits in two  physically interesting cases: $\alpha=-1$ (resp.  $\alpha=2$) corresponding respectively  to a gravitational (resp. elastic) like potential.
	\begin{cor}
		For $\alpha=-1$, then we get
		\begin{equation}
		\iota_1(\gamma)=\begin{cases}
		0 & \textrm{ if } \quad (\xi,-1)\in\Omega_{2,0}^{+,+}\bigcup\Omega_{2,0}^{+,0}\\
		1 & \textrm{ if }\quad (\xi,-1)\in\Omega_{2,0}^{+,-}.
		\end{cases}
		\end{equation} 
		moreover the circular orbit is linearly stable if $(\xi,-1)\in  \Omega_2^{+,0}$ and linearly unstable for $(\xi,-1)\in\Omega_2\backslash\Omega_2^{+,0}$.
	\end{cor}
\begin{proof}
	Since $\alpha=-1$, the corresponding region is  $\Omega_2^+$ and so $d<0$. If, $(\xi,-1)\in\Omega_2^{+,\pm}$, then we have $cd+b^2\neq 0$. Using Equation~\eqref{eq:splitting-of fundamental-solution-d<0} there exists a nontrivial Jordan block. In particular, we get that the corresponding circular solution is linearly unstable. 
	
	If $(\xi,-1)\in\Omega_2^{+,0}$, then we have $cd+b^2=0$ and consequently  taking into account Equation~\eqref{eq:splitting-of fundamental-solution-d<0-simple}, we get that the corresponding circular solution is linear stable.
	
	The proof for the case $\alpha=2$  can be obtained by arguing as before and will be left to the reader. 
	\end{proof}


\subsection{Circular orbits on the hyperbolic plane}\label{sec:hyperbolic case}

This subsection is devoted to compute the $\iota_1$ index for circular orbits on the pseudo-sphere.  In this case
\begin{equation}
p(\xi)=\dfrac{2}{(1-\xi^2)^2}\qquad \textrm{ and }  \qquad q(\xi)= \ln^\alpha\left(\dfrac{1+\xi}{1-\xi}\right)\qquad \xi\in(0,1).
\end{equation}
We let  
\begin{equation}
G(\xi)\=\ln\left(\dfrac{1+\xi}{1-\xi}\right)\qquad \textrm{ and } G'(\xi)\=\dfrac{dG}{d\xi}(\xi).
\end{equation}
By a direct computations we have 
\begin{align}\label{eq:derivatives-p-and-q-hyper}
&p'(\xi)=\dfrac{8\xi}{(1-\xi^2)^3} &&
p''(\xi)=\dfrac{8+40\xi^2}{(1-\xi^2)^4}\\
& q'(\xi)=\dfrac{2\alpha\cdot G^{\alpha-1}(\xi)}{1-\xi^2} && 
q''(\xi)=\dfrac{4\alpha(\alpha-1)G^{\alpha-2}(\xi)+4\alpha\xi G^{\alpha-1}(\xi)}{(1-\xi^2)^2}.
\end{align}
\begin{note}
Abusing notation, let us now introduce the following notation similar to that of  Equation~\eqref{eq:simplfied-notations-p-q}, we have
\begin{align}
&p_0=\dfrac{2}{(1-\xi_0^2)^2}&& p'_0=\frac{8\xi_0}{(1-\xi_0^2)^3}&& p''_0=\frac{8+40\xi_0^2}{(1-\xi_0^2)^4}\\
&q_0=G^\alpha(\xi_0) &&q'_0=\dfrac
{2\alpha G^{\alpha-1}(\xi_0)}{1-\xi_0^2} &&
q''_0=\dfrac{4\alpha(\alpha-1) G^{\alpha-2}(\xi_0)+4\alpha\xi_0 G^{\alpha-1}(\xi_0)}{(1-\xi_0^2)^2}\\
& \eta_0=\dfrac{2\xi_0^2}{(1-\xi_0^2)^2} && \eta_0'=\dfrac{4\xi_0^3+4\xi_0}{(1-\xi_0^2)^3}&& \dot \vartheta^2=\dfrac{-\alpha G^{\alpha-1}(\xi_0)(1-\xi_0^2)^2}{\xi_0(1+\xi_0^2)}.
\end{align}
\end{note}
Since the (RHS) of the equation defining $\dot \vartheta^2$ should be positive,  we  only restrict to the case 
\begin{equation}\label{eq:range-of-alpha-hyper}
\alpha<0.
\end{equation}
By a straightforward computation, we get 
\begin{multline}\label{eq:for-r11-hyper}
\zeta_0=\dfrac{4\xi_0(1+\xi_0^2)}{(1-\xi_0^2)^2}\cdot \sqrt{\dfrac{-\alpha G^{\alpha-1}(\xi_0)}{\xi_0(1+\xi_0^2)}} \qquad
\dfrac12 \eta_0''=\dfrac{6\xi_0^4+16\xi_0^2+2}{(1-\xi_0^2)^4}\qquad 
2q_0'\,(\ln \eta_0)'=\dfrac{8\alpha(\xi_0)(1+\xi_0^2)}{\xi_0(1-\xi_0^2)^2}
\\
\eta_0' \cdot \left[\dfrac{q_0'}{\eta_0'}\right]'=\dfrac{-2\alpha G^{\alpha-2}(\xi_0)[(\xi_0^4+6\xi_0^2+1)\cdot G(\xi_0)-2(\alpha-1)\xi_0(1+\xi_0^2)]}{\xi_0(1-\xi_0^2)^2(1+\xi_0^2)}.
\end{multline}
Then the four constants appearing at \eqref{eq:notations-a-b-c-d} are the following
\begin{multline}\label{eq:a-b-c-d-hyper}
a=p_0^{-1}=\dfrac{(1-\xi_0^2)^2}{2}\qquad 
b=\eta_0^{-1}\zeta_0=\dfrac{2(1+\xi_0^2)}{\xi_0}\cdot \sqrt{\dfrac{-\alpha G^{\alpha-1}(\xi_0)}{\xi_0(1+\xi_0^2)}}\qquad 
c=\eta_0^{-1}=\dfrac{(1-\xi_0^2)^2}{2\xi_0^2}\\
d=2q_0'\,(\ln \eta_0)'+\eta_0'\left(\dfrac{q_0'}{\eta_0'}\right)'=\dfrac{2\alpha G^{\alpha-2}(\xi_0)\{(3\xi_0^4+2\xi_0^2+3)\cdot G(\xi_0)+2(\alpha-1)\xi_0(1+\xi_0^2)\}}{\xi_0(1-\xi_0^2)^2(1+\xi_0^2)}.
\end{multline}
Similar to the sphere case, we start determining the sign of $d$ ranging in the  parameter region 
\begin{equation}
\Omega_3\=(0,1)\times (-\infty, 0).
\end{equation}
By the explicit computation of $d$ given at Equation~\eqref{eq:a-b-c-d-hyper}, we infer that  the sign of $d$ is  minus  the sign of $(3\xi_0^4+2\xi_0^2+3)\cdot G(\xi_0)+2(\alpha-1)\xi_0(1+\xi_0^2)$.
We let
\begin{equation}\label{eq:g-1-function}
g_1(\xi)=(3\xi^4+2\xi^2+3)\cdot \ln\left(\dfrac{1+\xi}{1-\xi}\right)+2(\alpha-1)\xi(1+\xi^2)\qquad (\xi,\alpha)\in\Omega_3.
\end{equation}
As before we split $\Omega_3$ into three subregions
\begin{equation}
\Omega_3=\Omega_3^+\bigcup \Omega_3^-\bigcup\Omega_3^0
\end{equation}
where
\begin{equation}
\Omega_3^\pm\=\Set{(\xi,\alpha)\in\Omega_3|\pm g_1(\xi)>0 },\quad  \Omega_1^0\=\Set{(\xi,\alpha)\in\Omega_3| g_1(\xi)=0 }.
\end{equation} 
By an algebraic manipulation, we get that 
\begin{multline}
\Omega_3^0=\Set{(\xi,\alpha)\in\Omega_3|\alpha=1-\dfrac{(3\xi^4+2\xi^2+3)\ln\left((1+\xi)/(1-\xi)\right)}{2\xi(1+\xi^2)}}\\
\Omega_3^{\pm}=\Set{(\xi,\alpha)\in\Omega_3|\pm\alpha>1-\dfrac{(3\xi^4+2\xi^2+3)\ln\left((1+\xi)/(1-\xi)\right)}{2\xi(1+\xi^2)}}.
\end{multline}
\paragraph{Case: $(\xi_0,\alpha)\in\Omega_3^-$.} Then we have $g_1(\xi_0)<0$ and consequently by Equation~\eqref{eq:a-b-c-d-hyper}, it follows that $d>0$. Since $c$ is positive, it readily follows that $b^2+cd>0$. By invoking  Equation~\eqref{eq:index-d-positive}, we finally get
\begin{equation}
\iota_1(\gamma)=0,\quad \text{if}\quad (\xi,\alpha)\in\Omega_3^-.
\end{equation}
\paragraph{Case: $(\xi_0,\alpha)\in\Omega_3^0$.} Then we have $g_1(\xi_0)=0$ and consequently by Equation~\eqref{eq:a-b-c-d-hyper}, we get that $d=0$. Since $c$ is positive, then by invoking Equation~\eqref{eq:index-d-zero-b-zero}  and  Equation~\eqref{eq:index-for-d-zero-b-nonzero}, we also have 
\begin{equation}
\iota_1(\gamma)=0,\quad \text{if}\quad (\xi,\alpha)\in\Omega_3^0.
\end{equation}
\paragraph{Case: $(\xi_0,\alpha)\in\Omega_3^+$.} Then we have $g_1(\xi_0)>0$ and consequently by Equation~\eqref{eq:a-b-c-d-hyper}, we get that $d<0$. 

 Again we have to study the sign of $b^2+cd$.  By Equation~\eqref{eq:a-b-c-d-hyper} and by a  direct computation, we get
\begin{equation}\label{eq:b-square-plus-cd-hyper}
b^2+cd=\frac{-\alpha G^{\alpha-2}(\xi_0)\{ (\xi_0^4+6\xi_0^2+1)\cdot G(\xi_0)-2(\alpha-1)\xi_0(1+\xi_0^2)\}}{\xi_0^3(1+\xi_0^2)}.
\end{equation}
Now, we observe that since $\alpha<0$ and $G(\xi_0)>0$,  then the sign of $b^2+cd$ is equal to  the sign of the function
\begin{equation}\label{eq:g-2-function}
g_2(\xi):=(\xi^4+6\xi^2+1)\ln\left(\dfrac{1+\xi}{1-\xi}\right)-2(\alpha-1)\xi(1+\xi^2)
\end{equation}
which is always positive in $\Omega_3^+$, being  $\alpha$ is negative.
\paragraph{Finding  the time interval.}  In order to compute the index by \eqref{eq:index-d-negative-and-zero-not-semisimple} we have to determine the value of $k$. Recall that the value of $k$ is determined by $k\cdot 2\pi<\sqrt{-ad}T\le (k+1)\cdot 2\pi$. Indeed,  we have 
\begin{equation}
T=\frac{2\pi}{\dot{\vartheta}}=2\pi\sqrt{\frac{\xi_0(1+\xi_0^2)}{-\alpha G^{\alpha-1}(\xi_0)(1-\xi_0^2)^2}}.
\end{equation}
Moreover, by Equation~\eqref{eq:a-b-c-d-hyper} we have
\begin{equation}
\sqrt{-ad}=\sqrt{-\dfrac{\alpha G^{\alpha-2}(\xi_0)(3\xi_0^4+2\xi_0^2+3)\cdot G(\xi_0)+2(\alpha-1)\xi_0(1+\xi_0^2)}{\xi_0(1+\xi_0^2)}}.
\end{equation}
Therefore, 
\begin{equation}
\sqrt{-ad}\cdot T=2\pi\cdot \sqrt{\dfrac{(3\xi_0^4+2\xi_0^2+3)\cdot G(\xi_0)+2(\alpha-1)\xi_0(1+\xi_0^2)}{G(\xi_0)(1-\xi_0^2)^2}}.
\end{equation}
Let
\begin{equation}\label{eq:g-3-function}
g_3(\xi):=\sqrt{\dfrac{(3\xi^4+2\xi^2+3)\cdot \ln\left(\dfrac{1+\xi}{1-\xi}\right)+2(\alpha-1)\xi(1+\xi^2)}{(1-\xi^2)^2\ln\left(\dfrac{1+\xi}{1-\xi}\right)}}.
\end{equation}

It is easy to see that  for every fixed $\alpha$ and  $k\in\N$ there exists $(\xi,\alpha)\in\Omega_3^+$ such that $k<g_3(\xi)\le k+1$. Since for every $\alpha<0$ 
we have 
\begin{equation}
\lim_{\xi\rightarrow 1^-}g_3(\xi)=+\infty
\end{equation}
we can define the subregions
\begin{equation}
	\Omega_{3,k}^{+}\=\Set{ (\xi,\alpha)\in\Omega_3^{+}|  k<g_3(\xi)\le k+1 }.
\end{equation}
All the subregions are shown in Figure \ref{fig:hyper-negative-alpha}. 
\begin{figure}[ht]
	\centering
	\includegraphics[scale=0.15]{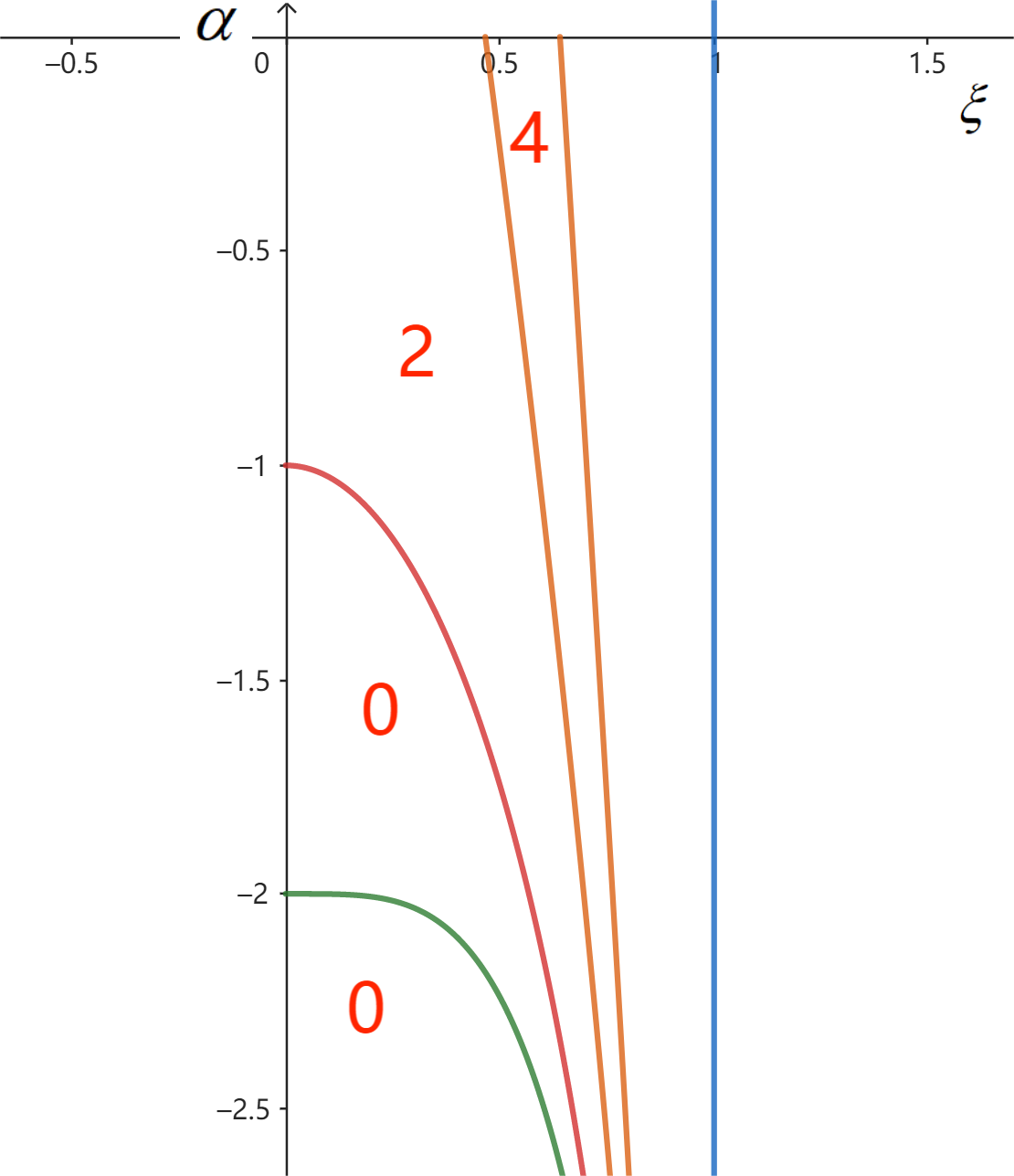}
		\caption{In the figure we represent the subregions of the $\Omega_3$ corresponding to the jumps of the Morse index.}
	\label{fig:hyper-negative-alpha}
\end{figure}
Invoking Equation~\eqref{eq:index-d-negative-and-zero-not-semisimple}, then we get that the $\iota_1$ index is given by 
\begin{equation}
\iota_1(\gamma)=2k \qquad \textrm{ if }\quad (\xi,\alpha)\in \Omega_{3,k}^+.
\end{equation}
Summing up, we conclude with the following result. 
\begin{thm} \label{pro:index-of-hyper}
	Under above notations, the Morse index of the circular orbit on the  hyperbolic plane is given by
	\begin{equation}
	\iota_1(\gamma)=\begin{cases}
	0 &\textrm{ if }\quad (\xi_0,\alpha)\in \Omega_{3}^{-}\bigcup\Omega_{3}^{0}\bigcup\Omega_{3,0}^+\\[7pt]
	2k  &\textrm{ if }\quad (\xi_0,\alpha)\in \Omega_{3,k}^+,\quad  k\in \N^*.
	\end{cases}
	\end{equation}	  
\end{thm}
As direct consequence of Proposition~\ref{pro:index-of-hyper}, in the special case of $\alpha=-1$  we get the following result. 
\begin{cor}
For  $\alpha=-1$, the Morse index of the circular solution is given by 
\begin{equation}
\iota_1(\gamma)=2k \qquad \textrm{ if }\qquad (\xi_0,-1)\in\Omega_{3,k}^{+},\quad k \in \N^*
\end{equation} 
and moreover, the circular orbit is  linearly unstable.
\end{cor}
\begin{proof}
If $\alpha=-1$, we get that the pair $(\xi_0, \alpha)$ belongs to  $\Omega_3^+$. In particular, it follows that $d<0$. Moreover, by the   previous discussion, the term  $cd+b^2$ is always positive. By taking into account of Equation~\eqref{eq:splitting-of fundamental-solution-d<0}, there always exists a non-trivial Jordan block. In particular, we get that the circular orbit is always linearly unstable. 
\end{proof}


\subsection{Euclidean case}
The last case is provided by the Euclidean one. We start letting 
\begin{align}
&p(\xi)=1&& q(\xi)=\xi^{\alpha}
\end{align}
and by a  direct computation, we get  
\begin{align}
&p'(\xi)=p''(\xi)=0&&q'(\xi)=\alpha \xi^{\alpha-1}&&q''(\xi)=\alpha(\alpha-1)\xi^{\alpha-2}
\end{align}
By Equation~\eqref{eq:simplfied-notations-p-q}, we get 
\begin{align}\label{eq:derivatives-p-and-q-Euclidean}
&p_0=1&& p'_0=0&& p''_0=0\\
&q_0=\xi_0^\alpha &&q'_0=\alpha\xi_0^{\alpha-1} &&
q''_0=\alpha(\alpha-1)\xi_0^{\alpha-2}\\
& \eta_0=\xi_0^2 && \eta_0'=2\xi_0&& \dot \vartheta^2=-\alpha\xi_0^{\alpha-2}
\end{align}

Since the (RHS) of the equation defining $\dot \vartheta^2$ should be positive,  we  only consider the case 
\begin{equation}\label{eq:range-of-alpha-Euclidean}
\alpha<0.
\end{equation}
By  a direct computation, we get
\begin{align}
&\zeta_0=2\xi_0\cdot \sqrt{-\alpha\xi_0^{\alpha-2}}&& 2q_0'\,(\ln \eta_0)'=4\alpha\xi_0^{\alpha-2}&&
\eta_0' \cdot \left[\dfrac{q_0'}{\eta_0'}\right]'=\alpha(\alpha-2)\xi_0^{\alpha-2}.
\end{align}
By this, we get that the four constants appearing at Equation~\eqref{eq:notations-a-b-c-d} are the following
\begin{align}\label{eq:abcd-euclidean}
&a=1&&b=\frac{2}{\xi_0}\sqrt{-\alpha\xi_0^{\alpha-2}}\\
&c=\frac{1}{\xi_0^2}&& d=\alpha(\alpha+2)\xi_0^{\alpha-2}
\end{align}
Now, we observe that 
\begin{equation}
d \textrm{ is}: \begin{cases}
\textrm{positive} &\textrm{ if }\quad \alpha<-2\\
\textrm{null} &\textrm{ if }\quad \alpha=-2\\
\textrm{negative} &\textrm{ if }\quad \alpha>-2.
\end{cases}
\end{equation}
Being $c>0$,  using Equation~\eqref{eq:index-d-positive},  Equation~\eqref{eq:index-d-zero-b-zero} and finally Equation~\eqref{eq:index-for-d-zero-b-nonzero} we get that 
\begin{equation}
\iota_1(\gamma)=0 \quad \text{if}\quad \alpha\le -2.
\end{equation}
If $-2<\alpha<0$,  we need to determine the sign of the term $b^2+cd$. By a direct computing we get 
\begin{equation}
b^2+cd=\alpha(\alpha-2)\xi_0^{\alpha-4}>0
\end{equation}
and  taking into account Equation~\eqref{eq:derivatives-p-and-q-Euclidean} and Equation~\eqref{eq:abcd-euclidean} we finally get 
\begin{align}
&T=\dfrac{2\pi}{\dot{\vartheta}}=2\pi\sqrt{\dfrac{1}{-\alpha\xi_0^{\alpha-2}}}&&
\sqrt{-ad}=\sqrt{-\alpha(\alpha+2)\xi_0^{\alpha-2}}.
\end{align}
Therefore
\begin{equation}
\sqrt{-ad}\cdot T=2\pi\cdot \sqrt{\alpha+2}
\end{equation}
and using Equation~\eqref{eq:index-d-negative-and-zero-not-semisimple}, we get that  $k=0$ if $\alpha\in(-2,-1]$ and  $k=1$ if $\alpha\in(-1,0)$ and we  have
\begin{equation}
\iota_1(\gamma)=\begin{cases}
0 &\textrm{ if }\quad \alpha\in(-2,-1]\\
2  &\textrm{ if }\quad \alpha\in (-1,0)
\end{cases}
\end{equation}

\begin{thm}\label{pro:index-of-euclidean}
	Under above notations, the Morse index  of a circular orbit in the  Euclidean plane  is given by
	\begin{equation}
	\iota_1(\gamma)=\begin{cases}
	0 &\textrm{ if }\quad \alpha\in(-\infty,-1]\\
	2  &\textrm{ if}\quad \alpha\in (-1,0)
	\end{cases} 
	\end{equation}	  
\end{thm}

By Proposition \ref{pro:index-of-euclidean}  we have the following direct corollary.
\begin{cor}
	For  $\alpha=-1$, we get
	\begin{equation}
	\iota_1(\gamma)=0
	\end{equation} 
and the corresponding circular orbit is  linearly unstable.
\end{cor}
\begin{proof}
	If $\alpha=-1$, then  we have $d<0$  and  $cd+b^2> 0$. By Equation~\eqref{eq:splitting-of fundamental-solution-d<0} there exists a nontrivial Jordan block. So, it is linearly unstable concluding the proof.
\end{proof}


\appendix

\section{A brief recap on the Maslov index}\label{sec:preliminary}

The aim of this section is to briefly recall the basic definitions, properties and main results that has beed used throughout the paper. Our basic references are \cite{Arn67, CLM94} and \cite{RS93}. 

\subsection{The Lagrangian Grassmannian and the Maslov  index}
We start by considering the Lagrangian Grassmannian manifold $Lag(2n)$, which is  the smooth manifold of all Lagrangian subspaces of the standard real symplectic space $(\R^{2n},\omega)$. For every $\Lambda_0\in Lag(2n)$, we define
\begin{equation}
\Omega^k(\Lambda_0):=\Set{\Lambda\in Lag(2n)|\dim(\Lambda\cap \Lambda_0)=k} \qquad  k=0,1,\ldots,n.
\end{equation}
We observe that $Lag(2n)=\bigcup^n_{k=0}\Omega^k(\Lambda_0)$. It's well known that $\Omega^k(\Lambda_0)$ is a connected embedded submanifold of $Lag(2n)$ of codimension $k(k+1)/2$ and, in particular, $\Omega^1(\Lambda_0)$ has codimension $1$. 

We denote by $\Sigma(\Lambda_0)$ the {\em Maslov cycle} also called by Arnol'd
in \cite{Arn67}  as the {\em train of $\Lambda_0$} and  defined by 
\[
\Sigma(\Lambda_0)=\bigcup^n_{k=1}\Omega^k(\Lambda_0).
\]
It is well-known that $\Omega^0(\Lambda_0)$ is open and dense in $Lag(2n)$ and that 
the Maslov cycle $\Sigma(\Lambda_0)$ is the closure  of $\Omega^1(\Lambda_0)$. Moreover,  $\Omega^1(\Lambda_0)$ is co-oriented in $Lag(2n)$ meaning that there exists a transverse orientation. More  precisely, for every $\Lambda\in\Omega^1(\Lambda_0)$, the smooth Lagrangian path $\Lambda(t)=\{e^{tJ}\Lambda, t\in[-\varepsilon,\varepsilon]\}$ crosses $\Omega^1(\Lambda_0)$ transversally for sufficiently small $\varepsilon>0$. The desired transverse positive orientation is defined by  the path $t \mapsto\Lambda(t)$ as $t$ runs from $-\varepsilon$ to $\varepsilon$. So,  $\Sigma(\Lambda_0)$ is two-sidedly embedded in $Lag(2n)$. By means of such an orientation, it is possible to define an intersection index between  a generic continuous path of Lagrangian subspaces and $\Sigma(\Lambda_0)$.
 \begin{defn}\label{def:definition of maslov index}
Let $t \mapsto\Lambda(t)$ be a continuous path in $Lag(2n)$ and let $\Lambda_{0}\in Lag(2n)$.  Then the {\em Maslov index} is defined as
\begin{equation}
\iCLM(\Lambda_{0},\Lambda(t)):=[e^{-\varepsilon J}\Lambda(t): \Sigma(\Lambda_{0})]
\end{equation}
where the right-hand side means the intersection number and $\varepsilon>0$ is sufficiently small.
\end{defn}
The Maslov index has several  properties which can be used in efficient way in its computation.  Here we just recall a couple of them that will be used in the paper. 
\begin{enumerate}
 \item[]{$\mathbf{Property \ I}$} {\em (Path additivity)} Let $c \in (a,b)$. Then, we 
 have \[
\iCLM(\Lambda_{0},\Lambda(t), t \in [a,b])=\iCLM(\Lambda_{0},\Lambda(t), t \in [a,c])+\iCLM(\Lambda_{0},\Lambda(t), t \in [c,b]).
\]
 \item[]{$\mathbf{Property \ II}$} {\em (Homotopy invariance with respect to ending points)} For a continuous family of Lagrangian
paths $\{\Lambda(s,t), s\in [0,1], t\in[a, b]\}$ such that $\dim (\Lambda(s, a)\cap \Lambda_0)$ and $\dim(\Lambda(s, b)\cap \Lambda_0)$ are constants, then
$\iCLM(\Lambda_0,\Lambda(0, t)) = \iCLM(\Lambda_0,\Lambda(1,t))$.
\end{enumerate}
Robbin and Salamon in \cite{RS93} provides an efficient formula for computing the Maslov index of a smooth path of Lagrangian subspaces by means  of the local contribution at each crossing instant. Let us consider the $\mathscr C^1$-path of  Lagrangian subspaces $\Lambda(t)(t\in[a,b])$. The instant $t_{0}$ is termed a {\em crossing instant} if $\Lambda(t_{0})\cap \Lambda_{0}\neq{0}$. 

Let $v$ be any vector in $\Lambda(t_{0})\cap \Lambda_{0}$ and $V_{t_{0}}$ be a fixed Lagrangian subspace which is transversal to $\Lambda(t_{0})$. For small $t$, the crossing form is defined by
\begin{equation}
\Gamma(\Lambda(t_{0}),\Lambda_{0},t_{0})(v)=\dfrac{d}{dt}\mid_{t=t_{0}}\omega(v,u(t))
\end{equation}
where $u(t)\in V_{t_{0}}$ such that $v+u(t)\in\Lambda(t)$. It is easy to check that this construction doesn't depend on  the choice of $V_{t_{0}}$. In the special case of Lagrangian path induced by a symplectic one by means of the transitive action of the symplectic group on the Lagrangian Grassmannian  $t\mapsto \Lambda(t):=\gamma(t)W$, where $\gamma(t)\in \Sp(2n)$ and $W$ is a fixed Lagrangian subspace, then the crossing form can be explicitly written by 
 \[
 \langle-\gamma(t)^{T}J\dot{\gamma}(t)v,v\rangle \textrm{  for } v\in \gamma(t)^{-1}(\Lambda(t)\cap W)
 \]
 or equivalently $\langle-J\dot{\gamma}(t)\gamma(t)^{-1}u,u\rangle$ for $u\in \Lambda(t)\cap W$ where $\langle\cdot,\cdot\rangle$
 stands for the standard inner product in $\R^{2n}$ .

A crossing is called {\em regular} if the crossing form is non-degenerate and a path is regular if every crossing is regular. Now, given a regular $\mathscr C^1$-path of Lagrangian subspaces, the following formula holds:
 \begin{equation}\label{eq:maslov index formula}
\iCLM(\Lambda_{0},\Lambda(t), t \in [a,b])=m^{+}(\Gamma(\Lambda(0),\Lambda_{0},a))+\sum_{t\in (a,b)} \sgn(\Gamma(\Lambda(t),\Lambda_{0},t))-
m^{-}(\Gamma(\Lambda(T),\Lambda_{0},b)),
\end{equation}
where the sum runs over the set of all crossings and $m^{+}$ (resp. $m^{-}$) denotes  the dimension of positive (resp.  negative) spectral subspace.
\paragraph{A special symplectic path.} Let us consider the  symplectic path 
\[
\gamma(t):=\begin{bmatrix}
M_{11}(t)&0\\M_{21}(t)&M_{22}(t)
\end{bmatrix}, t\in[a,b].
\]
In this case, there is a very useful formula for computing the Maslov index  of the Lagrangian path $t\mapsto \Lambda(t)$ pointwise defined through its graph. We refer the interested reader to \cite[Theorem 2.2]{Zhu06}. 

Let $V$ be a subspace of $\C^{2n}$ and we define
\begin{multline}\label{eq:app1}
V^I=\Set{x\in \C^{2n}|\omega(x,y)=0 \ \forall y\in V} \\
 W_I(V)=\Set{\trasp{(x,u,y,v)}\in\C^{4n}\ |\ \trasp{(x,y)}\in V^J, \trasp{(u,v)}\in V}.
\end{multline}
Then, we have 
\begin{equation}\label{eq:zhu's formula to compute maslov index}
\begin{aligned}
&\iCLM(W_I(V),Gr(\gamma(t)), t \in [a,b])\\
&=m^+(M_{1,1}(T)^*M_{2,1}(T)|_{S(T)})-m^+(M_{1,1}(0)^*M_{2,1}(0)|_{S(0)})+\dim S(0)-\dim S(T),
\end{aligned}
\end{equation}
where $m^+$ denotes the Morse positive index and $S(t)=\{x\in\C^n\ | \ \trasp{(x,M_{1,1}x)}\in V^I\}$. \footnote{This formula directly follows from   \cite[Theorem 2.2]{Zhu06} once setting  $K$ and $R$ as $I$ and $V$, respectively.}


\subsection{The generalized Conley-Zehnder index}

This section is devoted to recall the basic definitions, properties about the {\em generalized Conley-Zehnder index} that will be used throughout the paper. The basic references are \cite{HS09, LZ00a, LZ00b, Lon02}.

We start defining the one-codimensional subnmanifold 
\[ 
 \Sp(2n, \R)^{0}\=\{M\in \Sp(2n, \R)| \det(M-I_{2n})=0\}
\] 
of the Lie group $ Sp(2n, \R) $ and we let 
\[
\Sp(2n,\R)^{*}=\Sp(2n,\R)\backslash \Sp(2n, \R)^{0}=\Sp(2n,
\R)^{+}\bigcup
\Sp(2n, \R)^{-}
\]
where 
\begin{multline}
\Sp(2n,\R)=\{M\in \Sp(2n,\R)| \det(M-I_{2n})>0\} \textrm{ and }\\
\Sp(2n,\R)^{-}
=\{M\in \Sp(2n,\R)
|\det(M-I_{2n})<0\}.
\end{multline}
 For any $M\in \Sp(2n,\R)^{0}$ it can be proved that the manifold $\Sp(2n,\R)^{0}$ is
transversally oriented at any  point $M$ by choosing  as positive direction the one  determined by
 \[
 \dfrac{d}{dt}Me^{tJ}|_{t=0} \textrm{  with } t\geq0
 \]
 sufficiently small. 
 
 Now, we recall a useful binary operation usually referred to $\diamond$product of two matrices. Given  two $2m_{k}\times 2m_{k}$ matrices with the block form
$M_{k}=\begin{bmatrix} A_{k}&B_{k}\\C_{k}&D_{k}\end{bmatrix}$ with $k=1,2$, we
define the
 $\diamond$-product of $M_{1}$ and $M_{2}$ in the following way:
\[
M_{1}\diamond M_{2}\=\begin{bmatrix}
A_{1}&0&B_{1}&0\\0&A_{2}&0&B_{2}\\C_{1}&0&D_{1}&0\\0&C_{2}&0&D_{2}
\end{bmatrix}.
\]
\begin{defn}\cite[Definition 1,Pag.36 and Definition 5,Pag.38]{Lon02}\label{def:homotopy component} For every symplectic matrix $M\in\Sp(2n,\R)$ and $\omega\in \U$, we define 
\begin{equation}
\nu_\omega(M)=\dim_{\C}\ker_{\C}(M-\omega I_n)
\end{equation}
\begin{equation}
\Omega(M)=\{N\in\Sp(2n,\R)\ | \ \sigma(N)\cap\U=\sigma(M)\cap\U \textrm{ and } \nu_\omega(N)=\nu_\omega(M) \ \forall \ \omega\in\sigma(M)\cap\U \}. 
\end{equation}
 and we denote by $\Omega^0(M)$ the path connected component of $\Omega(M)$ which contains $M$. We refer to as the homotopy component of $M$ in $\Sp(2n,\R)$.
\end{defn}
We list below the basic normal forms corresponding to eigenvalues in $\U$ \cite[Definition 9, Pag.41]{Lon02}:
\begin{equation}\label{eq:basic normal forms}
\begin{aligned}
N_1(\lambda,b)&=\begin{bmatrix}
\lambda&b\\0&\lambda
\end{bmatrix},\ \lambda=\pm1, b=\pm1,0,\\
R(\vartheta)&=\begin{bmatrix}
\cos\vartheta&-\sin\vartheta\\\sin\vartheta&\cos\vartheta
\end{bmatrix}, \vartheta\in(0,\pi)\bigcup(\pi,2\pi),\\
N_2(\lambda,b)&=\begin{bmatrix}
R(\vartheta)&b\\0&R(\vartheta)
\end{bmatrix},\lambda=e^{i\vartheta}\ \text{and}\ \vartheta\in(0,\pi)\bigcup(\pi,2\pi),\\
b&=\begin{bmatrix}
b_1&b_2\\b_3&b_4
\end{bmatrix},b_j\in\R,b_2\neq b_3.
\end{aligned}
\end{equation}
The following result  will be used for investigating the stability problem.
\begin{lem}\cite[Theorem 10, Pag.41]{Lon02}\label{lem:homotopy component}
	For every symplectic matrix $M\in\Sp(2n,\R)$, there is a path $f:[0,1]\to \Omega^0(M)$ such that $f(0)=M$ and 
	\begin{equation}
	f(1)=M_1(\omega_1)\diamond\cdots\diamond M_k(\omega_k)\diamond M_0,
	\end{equation}
where each $M_j(\omega_j)$ is a basic normal form corresponding to some eigenvalue $\omega_j\in\U$ for $1\le j\le k$ and the symplectic matrix $M_0$ satisfies $\sigma(M_0)\cap \U=\emptyset$.
\end{lem}
 We denote the set of all symplectic paths $\gamma:[0,T]\to\Sp(2n,\R)$ such that $\gamma(0)=I_{2n}$ by  $\MP_T(2n)$. Following author in \cite{Lon02}, we recall
the following definition.
 \begin{defn}\label{def:Maslov-type index}
Let $\gamma\in\MP_T(2n)$. Then there exists an $\varepsilon>0$ such that $\psi(a)e^{J\vartheta}$ and $\psi(b)e^{J\vartheta}$ are not in $\Sp(2n,\R)^0$ for every $\vartheta\in(-\varepsilon,\varepsilon)\setminus\{0\}$. We define
\begin{equation}
\iota_1(\psi)\=[e^{-J\varepsilon}\psi:\Sp(2n,\R)^{0}],
\end{equation}
where the (RHS) denotes the intersection number between the perturbed path
$t\mapsto e^{-J\varepsilon} \psi(t)$
with the singular cycle $\Sp(2n,\R)^0$. We refer to $\iota_1$ as the {\em generalized Conley-Zehnder index}.
\end{defn}

The generalizer Conley-Zehnder index has several properties. Below we list a couple of them that will be useful in the paper and we refer the interested reader to  \cite[Corollary 6.2.5, Theorem 6.2.6]{Lon02}, for further details. 
\begin{lem}\label{lem:property of maslov-type index}
	\begin{enumerate}
		\item
		Given a symplectic path $\gamma:[0,T]\to\Sp(2n,\R)$ such that $\gamma(0)=I_{2n}$ and a matrix $M\in\Sp(2n,\R)$ we define $\beta(t) =
		M^{-1}\gamma(t)M$. Then we have 
		\[
		\iota_1(\gamma)=\iota_1(\beta).
		\]
			\item  For $j=1,2$ and symplectic path $\gamma_j:[0,T]\to\Sp(2n,\R)$ such that $\gamma_j(0)=I_{2n}$, then we have
		\begin{equation}
		\iota_1(\gamma_1\diamond\gamma_2)=\iota_1(\gamma_1)+\iota_1(\gamma_2).
		\end{equation}
	\end{enumerate}
\end{lem}
The next result establish the precise relation intertwining between  the $\iCLM$ and the $\iota$ indices and we refer the reader to  \cite[Corollary 2.1]{LZ00b} for its proof. 
\begin{lem}\label{lem:relation between two indices}
Let $\gamma:[0,T]\to \Sp(2n,\R)$ be a symplectic path such that $\gamma(0)=I_{2n}$, then we have
\begin{equation}
\iota_1(\gamma(t), t \in [0,T])+n=\iCLM(\Delta,Gr(\gamma(t)), t\in [0,T]),
\end{equation}
where $\Delta$ is the diagonal $Gr(I_{2n})$.
\end{lem}

\begin{lem}\label{lem:index of three special paths}

Let $\eta_1, \eta_2$ adn $\eta_3$ be the symplectic path pointwise defined by 
\[
\eta_1(t)=\begin{bmatrix}
1&t\\0&1
\end{bmatrix}\qquad \eta_2(t)=\begin{bmatrix}
1&-t\\0&1
\end{bmatrix}\quad 
\textrm{ and } \quad \eta_3(t)=\begin{bmatrix}
\cos t&-\sin t\\\sin t&\cos t
\end{bmatrix} \textrm{ for } t\in[0,T].
\]
Then we have
\begin{align}
& \iota_1(\eta_1)=-1 && \iota_1(\eta_2)=0\\
&
\iota_1(\eta_3)=
\begin{cases}
	1 &  \textrm{ if } \quad T\in(0,2\pi] \\
	3 &  \textrm{ if }\quad T\in(2\pi,4\pi]
\end{cases}
\end{align}
\end{lem}
\begin{proof}
The proof of this result, follows by Lemma \ref{lem:relation between two indices} 
and by the computation provides in \cite{KOP21}.  
\end{proof}


\subsection{An index theorem}

Let $T>0$ and we let $L\in \mathscr C^2(\R/T\ZZ\times \R^{2n}\R)$ be a Lagrangian function and we assume that it satisfies the  Legendre convexity condition
\[
L_{vv}(t,x,v)>0 \qquad \forall\, (t,x,v) \in [0, T ] \times T\R^n.
\]
On the Hilbert manifold $H=\{u\in  W^{1,2}(\R/T\ZZ,\R^n)\ | u(0)=u(T)\}$, we consider the Lagrangian action functional 
\begin{equation}\label{eq:abstract lagrangian system}
\mathcal{A}(x)=\int_{0}^{T}L(t,x,\dot x) dt.
\end{equation}
Then $x$ is a critical point of action functional \eqref{eq:abstract lagrangian system} if and only if $x$ satisfies following Euler-Lagrangian equation
\begin{equation}
\dfrac{d}{dt} L_v(t,x,\dot x)=L_x(t,x,\dot x)
\end{equation}
with boundary condition $x(0)=x(T)$ and $ \dot x(0)=\dot x(T)$. We conside the {\em index form} arising by the second variation of the action functional 
\begin{equation}
I(u,v)=\int_0^T \langle P\dot u,\dot v\rangle+\langle Qu,\dot v\rangle+\langle \trasp{Q}\dot u, v\rangle+\langle Ru, v\rangle dt
\end{equation}
where $u,v\in H$, $P(t)=L_{vv}(t,x,\dot x), Q(t)=L_{vx}(t,x,\dot x), R(t)=L_{xx}(t,x,\dot x)$ and we define  the {\em Morse index of the critical point $x$} as the dimension of the maximal subspace of $H$ where  $I$ is negative  definite. We denote Morse index by $m^-(x)$.

Integrating by parts in the index form, we get that  $x\in\ker I$ if and only if $x$ is a solution of following Sturm-Liouville boundary value problem
\begin{equation}\label{eq:S-L system}
\begin{cases}
-\dfrac{d}{dt}(p\dot x+Qx)+\trasp{Q}\dot x+Rx=0\qquad \textrm{ on } \qquad [0,T]\\[7pt]
x(0)=x(T) \qquad  x(0)=\dot x(T).
\end{cases}
\end{equation}
By setting  $y=P\dot x+Qx$ and $z=\trasp{(y,x)}$ then the equation defined at \eqref{eq:S-L system} corresponds to the Hamiltonian system
\begin{equation}\label{eq:Hamiltonian system}
\dot z(t)=JB(t)z(t) \textrm{ where }
B(t):=\begin{bmatrix} P^{-1}(t) &- P^{-1}(t)Q(t)\\-\trasp{Q}(t)P^{-1}(t)&\trasp{Q}(t)P^{-1}(t)Q(t)-R(t) \end{bmatrix}.
\end{equation} 
Denoting by $\gamma$ be the fundamental solution of the Hamiltonian system defined at Equation~\eqref{eq:Hamiltonian system}, then the following index theorem holds. (Cfr. \cite{Vit88, LA97}).
\begin{thm}\label{lem:relation between morse and maslov-type}
	Under the above notations, there following result holds:
	\begin{equation}
	m^-(x)=\iota_1(\gamma).
	\end{equation}
\end{thm}



\appendix


\vspace{1cm}
\noindent
\textsc{Dr. Stefano Baranzini}\\
Dipartimento di Matematica “Giuseppe Peano”\\
Università degli Studi di Torino\\
Via Carlo Alberto, 10 \\
10123, Torino\\
Italy\\
E-mail: \email{stefano.baranzini@unito.it}

\vspace{1cm}
\noindent
\textsc{Prof. Alessandro Portaluri}\\
Department of Agriculture, Forest and Food Sciences\\
Università degli Studi di Torino\\
Largo Paolo Braccini 2 \\
10095 Grugliasco, Torino\\
Italy\\
Website: \url{aportaluri.wordpress.com}\\
E-mail: \email{alessandro.portaluri@unito.it}

\vspace{1cm}
\noindent
\textsc{Dr. Ran Yang}\\
School of Science\\
East China University of Technology\\
Nanchang, Jiangxi, 330013\\
The People's Republic of China\\
China\\
E-mail: \email{201960124@ecut.edu.cn}


\begin{thebibliography}{99}


\bibitem[Arn67]{Arn67}
{\sc Arnol'd, V. I.}
\newblock On a characteristic  class entering into conditions of
quantization.
\newblock (Russian) Funkcional. Anal. i Priložen. 1 1967 1--14.

\bibitem[BJP16]{BJP16}
{\sc Barutello, Vivina; Jadanza, Riccardo D.; Portaluri, Alessandro}
\newblock  Morse index and linear stability of the Lagrangian circular orbit in a three-body-type problem via index theory.
\newblock Arch. Ration. Mech. Anal. 219 (2016), no. 1, 387--444.


\bibitem[BK17]{BK17}
{\sc Bolotin, S. V.; Kozlov, V. V.}
\newblock Topological approach to the generalized n-centre problem. 
\newblock Uspekhi Mat. Nauk 72 (2017), no. 3(435), 65--96; translation in Russian Math. Surveys 72 (2017), no. 3, 451--478.


\bibitem[CLM94]{CLM94}
{\sc Cappell, Sylvain E.; Lee, Ronnie; Miller, Edward Y.}
\newblock On the Maslov index.
\newblock Comm. Pure Appl. Math. 47 (1994), no. 2, 121--186.




\bibitem[CLW94]{CLW94}
{\sc Chow,S.; Li,C.; Wang, D.}
\newblock Normal forms and bifucations of planar vector fields
\newblock Cambridge University Press. Cambridge (1994)


\bibitem[DDZ19]{DDZ19}
{\sc  Deng, Yanxia; Diacu, Florin; Zhu, Shuqiang}
\newblock  Variational property of periodic Kepler orbits in constant curvature spaces. 
\newblock J. Differential Equations 267 (2019), no. 10, 5851--5869.

\bibitem[DPS12a]{DPS12a}
{\sc   Diacu, Florin; Pérez-Chavela, Ernesto; Santoprete, Manuele}
\newblock  The n-body problem in spaces of constant curvature. Part I: Relative equilibria. 
\newblock J. Nonlinear Sci. 22 (2012), no. 2, 247--266. 


\bibitem[DPS12b]{DPS12b}
{\sc 
 Diacu, Florin; Pérez-Chavela, Ernesto; Santoprete, Manuele}
 \newblock  The n-body problem in spaces of constant curvature. Part II: Singularities. 
 \newblock J. Nonlinear Sci. 22 (2012), no. 2, 267--275.
 
 
\bibitem[Fio17]{Fio17} 
{\sc Fiori, Simone}
\newblock Nonlinear damped oscillators on Riemannian manifolds: numerical simulation. 
\newblock Commun. Nonlinear Sci. Numer. Simul. 47 (2017), 207--222.

\bibitem[Gor77]{Gor77}
{\sc  Gordon, William B.}
\newblock  A minimizing property of Keplerian orbits. 
\newblock Amer. J. Math. 99 (1977), no. 5, 961--971.

%
\bibitem[HK92]{HK92}
{\sc Kozlov, Valeri V.; Harin, Alexander O.}
\newblock Kepler's problem in constant curvature spaces. 
\newblock Celestial Mech. Dynam. Astronom. 54 (1992), no. 4, 393--399. 

\bibitem[HS09]{HS09}
{\sc Hu, Xijun; Sun, Shanzhong}
\newblock Index and stability of symmetric periodic
orbits in Hamiltonian systems with application to figure-eight orbit.
\newblock Comm. Math. Phys. 290 (2009), no. 2, 737--777.

\bibitem[KOP21]{KOP21}
{\sc Kavle, Henry; Offin, Daniel; Portaluri, Alessandro}
\newblock  Keplerian orbits through the Conley-Zehnder index.
\newblock  Qual. Theory Dyn. Syst. 20 (2021), no. 1, Paper No. 10, 27 pp. 

\bibitem[Lon02]{Lon02}
{\sc  Long, Yiming}
\newblock Index theory for symplectic paths with applications.
\newblock Progress in Mathematics, 207. Birkh\"auser Verlag, Basel, 2002.

\bibitem[LA97]{LA97}
{\sc  Long, Yiming, An, Tianqing}
\newblock Indexing domains of instability for Hamiltonian systems
\newblock Nankai Inst. of Math.Preprint,1996(Revised 1997). NoDEA.5(1998),461-478.

\bibitem[LZ00a]{LZ00a}
{\sc  Long, Yiming; Zhu, Chaofeng}
\newblock Maslov-type index theory for symplectic paths and spectral flow.
I.
\newblock Chinese Ann. Math. Ser. B 20 (1999), no. 4, 413--424.



\bibitem[LZ00b]{LZ00b}
{\sc Long, Yiming; Zhu, Chaofeng}
\newblock Maslov-type index theory for symplectic paths
and spectral flow. II
\newblock  Chinese Ann. Math. Ser. B 21 (2000), no. 1, 89--108.
\bibitem[Mon13]{Mon13}
{\sc Montgomery, Richard}
\newblock MICZ-Kepler: dynamics on the cone over $\SO(n)$ 
\newblock Regul. Chaotic Dyn. 18 (2013), no. 6, 600--607. 

\bibitem[RS93]{RS93}
{\sc Robbin, Joel; Salamon, Dietmar}
\newblock The Maslov index for paths.
\newblock Topology 32 (1993), no. 4, 827--844.


\bibitem[Vit88]{Vit88}
{\sc  Viterbo, C.}
\newblock Indice de Morse des points critiques obtenus par minimax.
\newblock Ann.Inst.H.Poincar\'{e} Anal. non lin\'{e}aire. 5(1998), 221-226.


\bibitem[Zhu06]{Zhu06}
{\sc Zhu,Chaofeng}
\newblock A generalized Morse index theorem.
\newblock Analysis, World Sci. Publ., Hackensack, NJ, (2006), 493-540.

\end{thebibliography}
\end{document}